\documentclass[reqno]{amsart}

\usepackage{tsarxiv}

\begin{document}

\title{Affine processes beyond stochastic continuity}
    \author[M. Keller-Ressel]{Martin Keller-Ressel}
		\author[T. Schmidt]{Thorsten Schmidt}
		\author[R. Wardenga]{Robert Wardenga}
		\address{Albert-Ludwigs University of Freiburg, Ernst-Zermelo-Str. 1, 79104 Freiburg,  
		Freiburg Research Institute of Advanced Studies (FRIAS), Germany, and 
  University of Strasbourg Institute for Advanced Study (USIAS), France. 
}
    \email{Thorsten.Schmidt@stochastik.uni-freiburg.de}
		\address{Dresden University of Technology, Zellescher Weg 12-14, 01069 Dresden, Germany.}
		\email{Martin.Keller-Ressel@tu-dresden.de}
    \email{Robert.Wardenga@tu-dresden.de}
    \thanks{We thank the participants of the Freiburg-Wien-Z\"urich-seminar for stimulating discussions and very helpful comments.}
    \date{December 21, 2018}

\maketitle

\vspace{2mm}


\begin{abstract}
In this paper we study time-inhomogeneous affine processes beyond the common assumption of stochastic continuity. 
In this setting times of jumps can be both inaccessible and predictable. To this end we develop a general theory of  finite dimensional affine semimartingales under very weak assumptions. We show that the corresponding semimartingale characteristics have affine form and that the conditional characteristic function can be represented with solutions to measure differential equations of Riccati type. We prove existence of affine Markov processes and affine semimartingales under mild conditions and elaborate on examples and applications including affine processes in discrete time.
\end{abstract}

\section{Introduction}

The importance of jumps at predictable or predetermined times is widely acknowledged in the financial literature, see for example \cite{Merton1974,GeskeJohnson84,BelangerShreveWong2004,Piazzesi2001,Piazzesi2005,KimWright,DuffieLando2001,Fama1970,MillerRock1985}. This is due to the fact that a surprisingly large amount of jumps or, more generally, rapid changes in stock prices or other financial time series occur in correspondence with announcements released at scheduled and hence predictable times (see, e.g., \cite{Johannes2004}).  A prominent example  is the jump of the EUR/GBP exchange rate on the 23rd of June in 2016 when it became clear that the British referendum on membership in the EU will come out in favor of Brexit. In addition, large jumps in stock prices frequently coincide with the release of quarterly reports or earnings announcements. (See Figure \ref{fig:db} for an example and \cite{DupireTalk2017} for further empirical support). Econometric models incorporating such jumps at predetermined times were studied and tested on market data in \cite{Piazzesi2001}, see also \cite{Piazzesi2010} and \cite{GehmlichSchmidt2015MF, FontanaSchmidt2016}.

While affine processes are a prominent model class for interest rates or stochastic volatility, they have only been considered under the assumption of stochastic continuity, which precludes jumps at predictable times. This assumption is dropped in this paper, and we study affine processes only under very mild assumptions, which allow for jumps to occur at both predictable and totally inaccessible times. 

The defining property of affine processes is the exponential affine form of the conditional characteristic function which allows for rich structural properties while retaining tractability due to the representation of the conditional characteristic function in terms of ordinary differential equations, the so called 'generalized Riccati equations'. In subsequent research further applications have been explored (e.g. \cite{KRPT2013, Keller-Ressel2010, Duffie05}) as well as extensions of the state space (e.g. \cite{CFMT2011, CKMT2016}) and most notably an extension to time-inhomogeneous affine processes in \cite{Filipovic05}.

In Remark 2.11 of \cite{Filipovic05} the author conjectures that his results can also be obtained on the level of semimartingales omitting the assumption of stochastic continuity. Here we confirm this conjecture by generalizing the result in \cite{Filipovic05} to affine semimartingales with singular continuous and discontinuous characteristics and only locally integrable parameters. This result is complemented by existence results for affine Markov processes and affine semimartingales under certain mild assumptions. Furthermore we provide a variety of examples and applications. In particular we propose an affine term-structure framework that allows for discontinuities at previously fixed time-points.

\begin{figure}[h] \label{fig:db}
	 \begin{overpic}[width=14cm, clip]{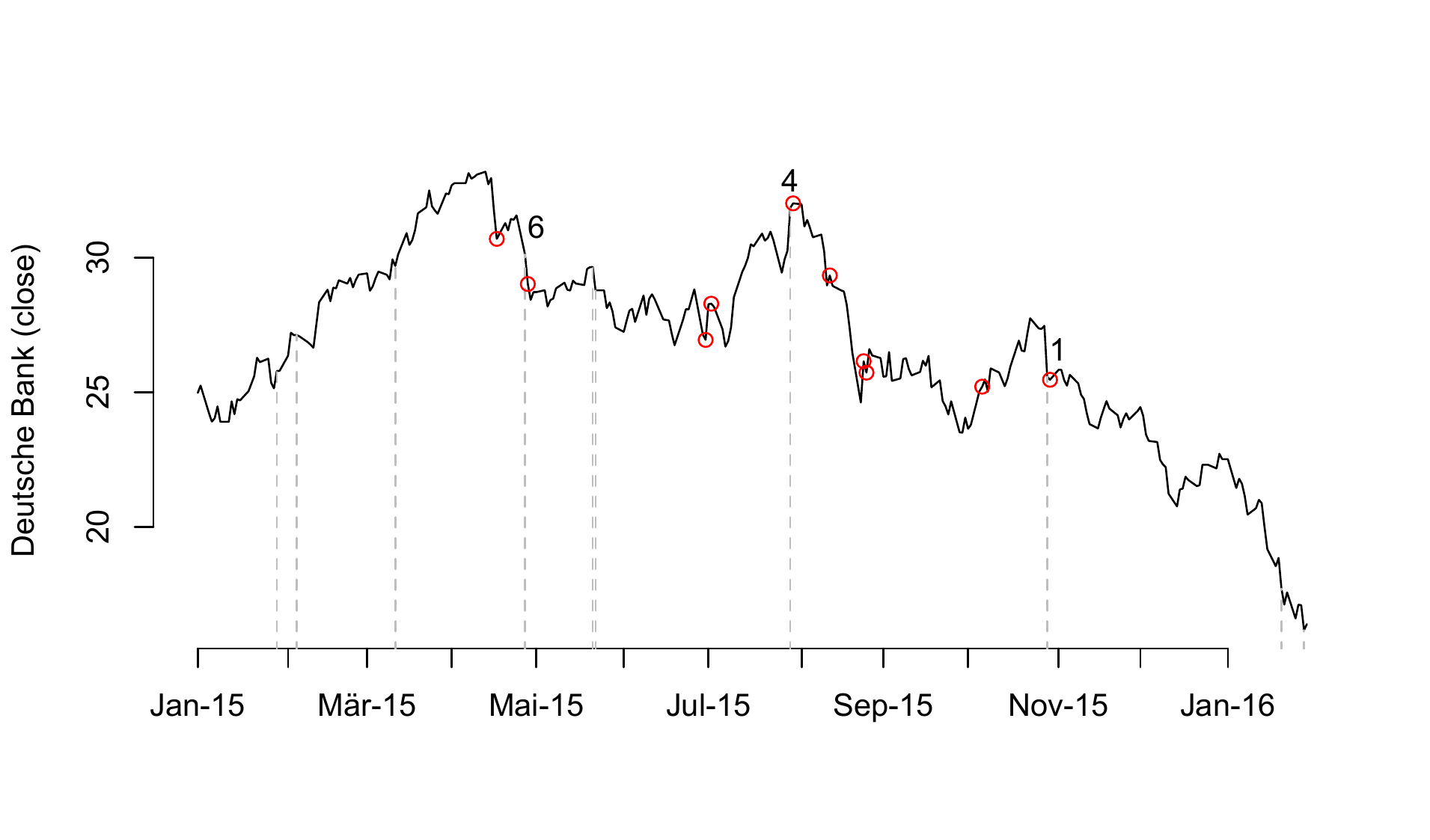} \end{overpic}  
	 \caption{Chart of the stock price of Deutsche Bank. The vertical lines represent dates which have been announced in the previous annual reports of 2013 and 2014, e.g. annual and quarterly reports and shareholder meetings. We marked the 10 largest one-day movements by circles; three (the largest, and the 4th- and 6th-largest) of them occurred at pre-announced dates.
	 }
\end{figure}

The paper at hand is structured as follows. The next section revisits some facts about semimartingales before stating the definition of \emph{affine semimartingales} and introducing certain technical assumption. After proving first results we define the concept of a \emph{good parameter set} in Section \ref{sec:characterization} which is a key ingredient of our first main result, the characterization Theorem \ref{prop1}. Section 4 discusses the relation between affine Markov processes and affine semimartingales as well as the important case of infinitely divisible processes. Section 5 is devoted to the existence of affine Markov processes and affine semimartingales under certain conditions on their good parameter set. Examples and applications are explained in Section 6 which concludes the paper with the introduction of a new \emph{affine term-structure framework}. Details about measure differential equations that appear in the characterization and existence results instead of the ODEs appearing in \cite{DuffieFilipovicSchachermayer} and \cite{Filipovic05}, are postponed to the appendix.

\section{Preliminaries}
\subsection{Affine Semimartingales}
Consider a filtered probability space $(\Omega,\cF,\bbF,P)$ with filtration $\bbF=(\cF_t)_{t \ge 0}$ satisfying the usual conditions.
A stochastic process $X$ taking values in $\R^d$ is called \emph{c\`adl\`ag}  if all its paths are right-continuous with left limits. For a c\`adl\`ag process $X$ we define $X_-$ and $\Delta X$ by 
\[\begin{cases}X_{0-} &= X_0, \quad  X_{t-} = \lim_{s \uparrow t} X_s \quad \text{ for }t>0,\\
\Delta X_t &= X_t - X_{t-}  . \end{cases}
\]
In particular, note that $\Delta X_0 = 0$ and that $X$ can be recovered from $X_-$ by taking right limits.

A \emph{semimartingale} is a process $X$ with decomposition $X=X_0+N+M$ where $X_0$ is $\cF_0$-measurable, $N$ is c\`adl\`ag, adapted, has paths of finite variation over each finite interval with $N_0=0$ and $M$ is a local martingale starting in $0$. We will always consider a c\`adl\`ag version of the semimartingale $X$. 

To the jumps of $X$ we associate an integer-valued random measure $\mu^X$ by 
	\begin{align} \label{def:muX}
		\mu^X(dt,dx) = \sum_{s \ge 0} \ind{\Delta X_s \neq 0} \delta_{(s,\Delta X_s)}(dt,dx); 
	\end{align}
here $\delta_a$ is the Dirac measure at point $a$. We denote the compensator, or the dual predictable projection, of the random measure $\mu^X$ by $\nu$. This is the unique predictable random measure which renders stochastic integrals with respect to $\mu^X-\nu$ local martingales. 

We briefly recall the well-known concept of \emph{characteristics} of a semimartingale, cf. \cite[Ch.~II]{JacodShiryaev}: a semimartingale $X$ with decomposition $X=X_0+N+M$ is called \emph{special} if $N$ is predictable. In this case, the decomposition  is unique, and we call it the \emph{canonical decomposition}. The local martingale part $M$ can be decomposed in a continuous local martingale part, which we denote by $X^c$, and a purely discontinuous local martingale part, $X-X^c$. We fix a truncation function $h:\R^d\to \R^d$ which is a bounded function satisfying $h(x)=x$ in a neighborhood of $0$. Then $\check X(h)=\sum_{s \le \cdot}( \Delta X_s - h(\Delta X_s) )$ and $X(h)=X-\check X(h)$ both define $d$-dimensional stochastic processes. Note that $\Delta X(h) = h(\Delta X)$, such that $X(h)$ has bounded jumps. The resulting process is a special semimartingale and we denote its canonical decomposition by
$$ X(h) = X_0 + B(h) + M(h), $$ 
with a predictable process of finite variation $B(h)$ and a local martingale $M(h)$.  
The \emph{characteristics} of the semimartingale $X$ is the triplet $(B,C,\nu)$ where $B=B(h)$, $C=(C^{ij})$ with $C^{ij}=\scal{X^{i,c}}{X^{j,c}}$ and $\nu=\nu^X$ is the compensator of $\mu^X$ defined in Equation \eqref{def:muX}. For additional facts on semimartingales and stochastic analysis we refer to \cite{JacodShiryaev}.

Let $D \subset \R^d$ be a closed convex cone of full dimension, i.e., a convex set, closed under multiplication with positive scalars, and with linear hull equal to $\R^d$. An important example is the set $\Rplus^m \times \R^n$ with $m+n=d$, which was used as the `canonical state-space' for affine processes in \cite{DuffieFilipovicSchachermayer,Filipovic05}. For $u,w$ in $\C^d$ we set $\scal{u}{w}=\sum_{i=1}^d u_i w_i$ and denote the real part of $u$ by $\Re u$. Moreover, we define the \emph{complex dual cone} of the state space $D$ by
\begin{equation}\label{eq:defU}
\cU := \set{u \in \C^d: \scal{\Re u}{x} \le 0  \text{ for all }  x \in D}.
\end{equation}
For the canonical state space $\cU$ equals $\C_{\le 0}^m \times i\R^n$, where $\C_{\le 0} = \set{u \in \C: \Re u \le 0}$, which coincides with the definition used in \cite{DuffieFilipovicSchachermayer}.\footnote{We use this notation in analogous fashion for $<$,$>$ or $\ge$ instead of $\le$ and with $\R$ instead of $\C$.} We are now prepared to state the central definition of this paper. 

\begin{definition}\label{def:affine}
Let $X$ be a c\`adl\`ag  semimartingale, taking values in $D$. The process $X$ is called an \emph{affine semimartingale}, if there exist $\C$ and $\C^d$-valued deterministic functions $\phi_s(t,u)$ and $\psi_s(t,u)$, continuous in $u \in \cU$ and with $\phi_s(t,0) = 0$ and $\psi_s(t,0)=0$,  such that 
\begin{align}\label{cond:affine}
E\big[ e^{\scal{u}{X_{t}}} |\cF_{s} \big] = \exp\big( \phi_s(t,u) + \scal{\psi_s(t,u)}{X_s} \big)
\end{align}
for all $0 \le s \le t$ and $u \in \cU$. Moreover, $X$ is called \emph{time-homogeneous}, if $\phi_s(t,u) = \phi_0(t-s,u)$ and $\psi_s(t,u)=\psi_0(t-s,u)$, again for all $0 \le s \le t$ and  $u \in \cU$. 
\end{definition}

Note that the left-hand side of \eqref{cond:affine} is always well-defined and bounded in absolute value by $1$, due to the definition of $\cU$.

\begin{remark}
Comparing Definition~\ref{def:affine} with the definition of an \emph{affine process} in \cite{DuffieFilipovicSchachermayer} (which treats the time-homogeneous case) and \cite{Filipovic05} (which treats the time-inhomogeneous case), we have replaced the Markov assumption of \cite{DuffieFilipovicSchachermayer,Filipovic05} with a semimartingale assumption. In view of \cite[Thm.~2.12]{DuffieFilipovicSchachermayer} this seems to slightly restrict the scope of the definition, since it excludes non-conservative processes. On the other hand, and this is the central point of our paper, we do not impose a stochastic continuity assumption on $X$, as has been done in \cite{DuffieFilipovicSchachermayer,Filipovic05}. It turns out that omitting this assumption leads to a significantly larger class of stochastic processes and to a substantial extension of the results in \cite{DuffieFilipovicSchachermayer,Filipovic05}.
\end{remark}

To continue, we introduce an important condition on the support of the process $X$. Recall that the support of a generic random variable $X$, is the smallest closed set $C$ such that $P(X\in C)=1$; we denote this set by $\supp(X)$. For a set $A$ we write $\conv (A)$ for its convex hull, i.e.\ the smallest convex set containing $A$.

\begin{condition}\label{def:full_support}
We say that an affine semimartingale $X$ has \emph{support of full convex span}, if $\conv(\supp(X_t)) = D$ for all $t >0$. 
\end{condition}

Under Condition~\ref{def:full_support}, $\phi$ and $\psi$ are uniquely specified:

\begin{lemma}\label{lem:uniqueness} Let $X$ be an affine semimartingale satisfying the support condition~\ref{def:full_support}. Then $\phi_s(t,u)$ and $\psi_s(t,u)$ are uniquely specified by \eqref{cond:affine} for all $0 < s \le t$ and $u \in \cU$.
\end{lemma}

\begin{proof}
Fix $0 < s \le t$ and suppose that $\wt \phi_s(t,u)$ and $\wt \psi_s(t,u)$ are also continuous in $u \in \cU$ and satisfy \eqref{cond:affine}. Write $p_s(t,u) := \wt \phi_s(t,u) - \phi_s(t,u)$ and $q_s(t,u) := \wt \phi_s(t,u) - \phi_s(t,u)$. Due to \eqref{cond:affine} it must hold that 
\[p_s(t,u) + \scal{q_s(t,u)}{X_s} \quad \text{takes values in } \quad \set{2 \pi i k: k \in \NN}\quad a.s. \;\forall\, u \in \cU.\]
However, the set $\cU$ is simply connected, and hence its image under a continuous function must also be simply connected. It follows that $u \mapsto p_s(t,u) + \scal{q_s(t,u)}{X_s}$ is constant on $\cU$ and therefore equal to $p_s(t,0) + \scal{q_s(t,0)}{X_s} = 0$. Hence, 
\[p_s(t,u) + \scal{q_s(t,u)}{x} = 0, \]
for all $x \in \supp(X_s)$ and $u \in \cU$. Taking convex combinations, the equality can be extended for $x\in D$. Since $D$ has full linear span, we conclude that 
$p_s(t,u) = 0$ and $q_s(t,u) = 0$ for all $u \in \cU$, completing the proof.
\end{proof}

\begin{definition}\label{def:quasi_regular}
An affine semimartingale is called \emph{quasi-regular}, if the following holds:
\begin{enumerate}[(i)]
\item The functions $\phi$ and $\psi$ are of finite variation in $s$ and c\`adl\`ag in both $s$ and $t$. More precisely, we assume that for all $(t,u) \in \Rplus \times \cU$
\[s \mapsto \phi_s(t,u) \quad \text{and} \quad s \mapsto \psi_s(t,u)\]
are c\`adl\`ag functions of finite variation on $[0,t]$, and for all $(s,u) \in \Rplus \times \cU$ 
\[t \mapsto \phi_s(t,u) \quad \text{and} \quad t \mapsto \psi_s(t,u)\]
are c\`adl\`ag functions on $[s, \infty)$.
\item For all $0 < s \le t$, the functions 
\[u \mapsto \phi_{s-}(t,u) \quad \text{and} \quad u \mapsto \psi_{s-}(t,u)\]
are continuous on $\cU$.
\end{enumerate}
\end{definition}

\begin{remark}
Definition~\ref{def:quasi_regular} should be compared to the assumptions imposed in \cite{DuffieFilipovicSchachermayer} and \cite{Filipovic05}. In both papers technical `regularity conditions' are defined. In \cite{DuffieFilipovicSchachermayer, Filipovic05} $\phi$ and $\psi$ are automatically continuous in their first argument, due to the stochastic continuity of $X$. In addition they are assumed continuously differentiable from the right, with a derivative that is continuous in $u$. Thus, (i) and (ii) are clearly milder than the regularity assumptions in \cite{DuffieFilipovicSchachermayer} or \cite{Filipovic05}.
\end{remark}


\subsection{First results on $\phi$ and $\psi$}

We proceed to show first analytic results on the functions $\phi$ and $\psi$ from \eqref{cond:affine}. 

\begin{lemma}\label{lem:flow_cont} Let $X$ be an affine semimartingale satisfying the support condition~\ref{def:full_support}. Then,
\begin{enumerate}[(i)]
\item  the function $u \mapsto \phi_s(t,u)$ maps $\cU$ to $\C_{\le 0}$ and $u \mapsto \psi(t,u)$ maps $\cU$ to $\cU$, for all $0<s\le t$,
\item $\phi$ and $\psi$ satisfy the \emph{semi-flow property}, i.e.  for all  $0 < s \le r \le t$ and $u \in \cU$, 
\begin{align}\label{eq:flow}
\begin{array}{r@{}lrl}
\phi_s(t,u) &= \phi_r(t,u) + \phi_s(r,\psi_r(t,u)), &\qquad \phi_t(t,u) &= 0\\
\psi_s(t,u) &= \psi_s(r,\psi_r(t,u)), &\qquad \psi_t(t,u) &= u.
\end{array}
\end{align}
\end{enumerate}
\end{lemma}

\begin{proof}
To show the first property,  recall that by Equation \eqref{cond:affine} we have 
\begin{equation}\label{eq:affine_repeat}
E\big[ e^{\scal{u}{X_{t}}} |\cF_{s} \big] = \exp\big( \phi_s(t,u) + \scal{\psi_s(t,u)}{X_s} \big)
\end{equation}
for all $u \in \cU$ and $0 \le s \le t$.
Since $\scal{\Re u}{X_{t}} \le 0$, a.s., the left hand side is bounded by one in absolute value. Thus, also
\[ \Re \phi_s(t,u) + \scal{\Re \psi_s(t,u)}{X_s} \le 0,\  a.s. \]
and consequently 
\[ \Re \phi_s(t,u) + \scal{\Re \psi_s(t,u)}{x} \le 0, \quad \text{for all $x \in \supp(X_s)$.} \]
Taking arbitrary convex combinations of these inequalities and using that $\conv(\supp(X_s)) = D$ by Condition~\ref{def:full_support}, we obtain that the inequality must in fact hold for all $x \in D$. Since $D$ is a cone this implies that $\Re \phi_s(t,u) \le 0$ and $\psi_s(t,u) \in \cU$, proving (i).

To show the semi-flow equations we apply iterated expectations to the left hand side of \eqref{eq:affine_repeat}, yielding 
\begin{align*}
E\big[ E\big[ e^{\scal{u}{X_t}} |\cF_{r} \big] |\cF_{s} \big] &= E\big[ \exp\big( \phi_r(t,u) + \scal{\psi_r(t,u)}{X_r} \big) |\cF_{s} \big] = \\
&= \exp\big( \phi_s(r,u) + \phi_s(r,\psi_r(t,u)) + \scal{\psi_s(r,\psi_r(t,u))}{X_s} \big).
\end{align*}
Note that the exponent on the right hand side is continuous in $u$ and that the same holds true for \eqref{eq:affine_repeat}. By the same argument as in the proof of Lemma~\ref{lem:uniqueness} we conclude that 
\[\phi_s(t,u) + \scal{\psi_s(t,u)}{x} = \phi_s(r,u) + \phi_s(r,\psi_r(t,u)) + \scal{\psi_s(r,\psi_r(t,u))}{x},\]
for all $x \in D$. Since the linear hull of $D$ is $\R^d$ the semi-flow equations \eqref{eq:flow} follow. Note that the terminal conditions $\psi_t(t,u) = u$ and $\phi_t(t,u) = 0$ are a simple consequence of $\cond{\exp(\scal{u}{X_t})}{\cF_t} = \exp(\scal{u}{X_t})$ and the uniqueness property from Lemma~\ref{lem:uniqueness}.
\end{proof}
\begin{remark}
Note that $s = 0$ is excluded from the semi-flow equations, since Condition~\ref{def:full_support}does not apply to the initial value $X_0$ of $X$. However, as soon as quasi-regularity is imposed, the c\`adl\`ag property of $\phi$ and $\psi$ immediately allows to extend the semi-flow equations also to $s= 0$.
\end{remark}
\begin{remark}\label{rem:big_flow}
To express the semi-flow equations in a more succinct matter, it is sometimes convenient to introduce the following `big-flow'-notation. Define the set $\wh{\cU} := \C_{\le 0}\times \cU$ and denote its elements by $\wh{u} = (u_0,u)$. Define 
\[\Psi_s(t, \wh{u}) := \begin{pmatrix}\phi_s(t,u) + u_0\\ \psi_s(t,u) \end{pmatrix}.\]
Part (i) of Lemma \ref{lem:flow_cont} is equivalent to the claim that $u \mapsto \Psi_s(t,u)$ maps $\wh{\cU}$ to $\wh{\cU}$ and part (ii) is equivalent to 
\[\Psi_s(t,\wh{u}) = \Psi_s(r,\Psi_r(t,\wh{u})), \qquad \Psi_t(t,\wh{u}) = \wh{u},\]
for all $0 < s \le r \le t$ and $\wh{u} \in \wh{\cU}$.
\end{remark}

\begin{lemma}\label{lem:affine_limits}
Let $X$ be a quasi-regular affine semimartingale. Then,
\begin{align}
E\big[ e^{\scal{u}{X_{t-}}} |\cF_{s} \big] &= \exp\big( \phi_{s}(t\uminus,u) + \scal{\psi_{s}(t\uminus,u)}{X_{s}} \big), \quad \forall\,0 \le s < t, u \in \cU.\label{eq:Utcond1}\\
E\big[ e^{\scal{u}{X_{t}}} |\cF_{s-} \big] &= \exp\big( \phi_{s-}(t,u) + \scal{\psi_{s-}(t,u)}{X_{s-}} \big), \quad \forall\,0 < s \le t, u \in \cU \label{eq:Utcond2}.
\end{align}
If in addition $X$ satisfies the support condition~\ref{def:full_support}, it also holds that 
\begin{equation}\label{eq:jump}
E\big[ e^{\scal{u}{ \Delta X_{t}}} |\cF_{t-} \big] = \exp\big(- \Delta \phi_{t}(t,u) - \scal{\Delta \psi_{t}(t,u)}{X_{t-}} \big), \quad \forall\; (t,u) \in \Rplus \times \cU.
\end{equation}
\end{lemma}
\begin{proof}
The first expression, \eqref{eq:Utcond1}, follows by taking left limits in $t$ on both sides of \eqref{cond:affine}. On the right hand side, the limit is well-defined by the c\`adl\`ag property of $\phi$ and $\psi$ in $t$. On the left hand side, dominated convergence and the c\`adl\`ag property of $X$ yield \eqref{eq:Utcond1}. Equation \eqref{eq:Utcond2} follows from a similar argument, now taking left limits in $s$. Indeed, note that for any integrable random variable $Y$ martingale convergence yields that that $\lim_{\epsilon \downarrow 0} E\big[Y |\cF_{s-\epsilon} \big] = E\big[Y |\cF_{s-} \big]$. Equation~\eqref{eq:jump} follows by evaluating \eqref{eq:Utcond2} at $s = t$ and noting that 
$\Delta \phi_t(t,u) = \phi_t(t,u) - \phi_{t-}(t,u) = -\phi_{t-}(t,u)$, and $\Delta \psi_t(t,u) = \psi_t(t,u) - \psi_{t-}(t,u) = u -\psi_{t-}(t,u)$, due to Lemma~\ref{lem:flow_cont}.
\end{proof}

\begin{lemma}\label{lem:double_limit}
Let $X$ be a quasi-regular affine semimartingale satisfying the support condition~\ref{def:full_support}. Then, 
\begin{enumerate}[(i)]
\item for all $(s,u) \in \Rplus \times \cU$ the functions
\[t \mapsto \phi_{s-}(t,u), \qquad t \mapsto \psi_{s-}(t,u)\]
are c\`adl\`ag on $[s, \infty)$.
\item The `double limits' $\phi_{s-}(t\uminus,u)$ and $\psi_{s-}(t\uminus,u)$ are well-defined and independent of the order of limits, i.e.,  
\[\lim_{\epsilon \downarrow 0} \psi_{s-}(t - \epsilon,u) = \lim_{\delta \downarrow 0} \psi_{s - \delta}(t-,u),\]
and similarly for $\phi$.
\item The semi-flow equations \eqref{eq:flow} still hold when $s$ is replaced by $s-$ or $t$ is replaced by $t-$ (or both).
\item It holds that 
\[E\big[ e^{\scal{u}{X_{t-}}} |\cF_{s-} \big] = \exp\big( \phi_{s-}(t\uminus,u) + \scal{\psi_{s-}(t\uminus,u)}{X_{s-}} \big), \label{eq:Utcond3}\]
for all $0 < s \le t$ and $u \in \cU$.
\item For all $u \in \cU$ and $0\leq s < t$ it holds that
\begin{align}\label{eq:computeDeltaphipsi}
\begin{split}
	\Delta \phi_s(t,u) & = \Delta \phi_s(s,\psi_s(t,u)), \\
	\Delta \psi_s(t,u) & = \Delta \psi_s(s,\psi_s(t,u)).
\end{split}
\end{align}
\end{enumerate}
\end{lemma}
\begin{proof}
We show claims (i), (ii) and (iii) for $\psi$ only. The proof can easily be extended to $\phi$, e.g. by using the `Big flow' argument of Remark~\ref{rem:big_flow}. To show right continuity in (i), we write 
\begin{align*}
\lim_{\epsilon \downarrow 0} \psi_{s-}(t+\epsilon,u) &= \lim_{\epsilon \downarrow 0} \psi_{s-}(t, \psi_t(t+\epsilon,u)) = \psi_{s-}\Big(t,  \lim_{\epsilon \downarrow 0} \psi_t(t+\epsilon,u)\Big) = \\ &= \psi_{s-}(t, \psi_t(t,u)) = \psi_{s-}(t,u).
\end{align*}
Here, we have used the flow property, the continuity of $\psi_{s-}(t,u)$ in $u$ and finally the right-continuity of $\psi_s(t,u)$ in $t$.
As for the left limit, the equality
\begin{align*}
\lim_{\epsilon \downarrow 0} \psi_{s-}(t-\epsilon,u) &= \lim_{\epsilon \downarrow 0} \psi_{s-}(s, \psi_s(t-\epsilon,u)) = \psi_{s-}\Big(s,  \lim_{\epsilon \downarrow 0} \psi_t(t-\epsilon,u)\Big) = \\ &= \psi_{s-}(s, \psi_s(t-,u))
\end{align*}
shows that the left limit exists. Moreover, 
\[\psi_{s-}(s, \psi_s(t-,u)) =  \lim_{\delta \downarrow 0} \psi_{s-\delta}(s, \psi_s(t-,u)) = \lim_{\delta \downarrow 0} \psi_{s-\delta}(t-,u))\]
shows exchangeability of the limits in (ii). Claim (iii) follows from the semi-flow equations \eqref{eq:flow} by taking left limits in $s$, left limits in $t$, or both.
Similarly, claim (iv) follows from \eqref{eq:Utcond1} by taking left limits in $s$, or from \eqref{eq:Utcond2} by taking left limits in $t$.

For (v) we apply the semi-flow property \eqref{eq:flow} for $r=s$ and obtain that
\begin{align*}
   \Delta \phi_s(t,u) &=  \phi_s(t,u) -  \phi_{s-}(t,u) 
   		= \phi_s(s,\psi_s(t,u)) -  \phi_{s-}(s,\psi_s(t,u)) 
\end{align*}
and the first part of \eqref{eq:computeDeltaphipsi} follows. The second part follows analogously.
 \end{proof}

\section{The characterization of affine semimartingales}\label{sec:characterization}

In this section we derive the representation of affine semimartingales via their semimartingale characteristics as well as generalized measure Riccati equations for the coefficients $\phi$ and $\psi$. It turns out that the class of affine semimartingales substantially generalizes the class of stochastically continuous affine processes: first, jumps at fixed time points are allowed and second, the jump height may depend on the state of the process.

Throughout, we will use the short-hand notation $\alpha=(\alpha_0,\bar \alpha)$ for a generic $d+1$-dimensional vector  $\alpha=(\alpha_0,\dots,\alpha_d)$. Moreover, we denote by $\cS^{d}_+$ the convex cone of symmetric positive semi-definite $d\times d$ matrices. Given characteristics $(B,C,\nu)$ of a semimartingale $X$, recall from \cite[Eq.~II.1.23, Prop.~II.2.6]{JacodShiryaev} that $C$ is always continuous and $B$ can be decomposed as $B = B^c + \sum \Delta B$. Furthermore, also a `continuous part' $\nu^c$ of $\nu$ can be defined by
\begin{equation}\label{eq:cont_nu}
\begin{split}
\cJ &:= \set{(\omega,t): \nu(\omega,\set{t},D) > 0}\\
\nu^c(\omega,dt,dx) &:= \nu(\omega,dt,dx) \bI_{\cJ^\complement}(\omega,t).
\end{split}
\end{equation}
Finally, if one chooses a `good version' (as we always do) of the characteristics, then
\begin{equation}\label{eq:B_jumps}
\Delta B_t = \int_D h(x) \nu(\set{t}, dx),
\end{equation}
where $h$ is the truncation function for the jumps; cf. \cite[Prop.~II.2.9]{JacodShiryaev}.
We introduce the following definition, which will be needed to formulate our main results.

\begin{definition}\label{def:good-parameter}
Let $A$ be a non-decreasing c\`adl\`ag function with continuous part $A^c$ and jump points $J^A\coloneqq\set{t\ge 0\vert \Delta A_t> 0}$. Let $(\gamma,\beta,\alpha,\mu)=\brac{\gamma_i,\beta_i,\alpha_i,\mu_i}_{i\in\setd}$ be functions such that $\gamma_0\colon \R_{\ge0}\times\cU\rightarrow\C$, $\bar{\gamma}\colon \R_{\ge0}\times\cU\rightarrow \C^d$, $\beta_i\colon\R_{\geq0}\rightarrow\R^d$, $\alpha_i\colon\R_{\ge0}\rightarrow\mathcal{S}^d$ and $\brac{\mu_i\brac{t,\cdot}}_{t\ge 0}$ are families of (possibly signed) Borel measures on $D\setminus\set{0}$.
We call $(A,\gamma,\beta,\alpha,\mu)$ a \emph{good parameter set} if for all $i\in\setd$,
\begin{enumerate}[(i)]
\item $\alpha_i$ and $\beta_i$ are locally integrable w.r.t. $A^c$,
\item for all compact sets $K\subset D\setminus\set{0}$, $\mu\brac{\cdot,K}$ is locally $A^c$-integrable.
\item$\gamma(t,u) = 0$ for all $(t,u) \in (\R_{\ge 0} \setminus J^A) \times \cU$. 
\end{enumerate}
\end{definition}

\begin{theorem}\label{prop1}
Let $X$ be a quasi-regular affine semimartingale satisfying the support condition~\ref{def:full_support}. Then there exists a good parameter set $(A,\gamma,\beta, \alpha, \mu)$ such that
the semimartingale characteristics $(B,C,\nu)$ of $X$ w.r.t. the truncation function $h$ satisfy, $\PP$-a.s. for any $t >0$, 
\begin{subequations}\label{char:affine}
\begin{align} 
B^c_t(\omega) &= \int_0^t \big( \beta_0 (s)  + \sum_{i=1}^d X^i_{s-}(\omega) \beta_i(s)   \big) dA^c_s \label{eq:char_B}\\
C_t(\omega) &= \int_0^t \big( \alpha_0 (s) + \sum_{i=1}^d X^i_{s-}(\omega) \alpha_i (s)  \big) dA^c_s \label{eq:char_C}\\
\nu^c(\omega,ds,dx) &= \big( \mu_0 (s,dx) + \sum_{i=1}^d X^i_{s-}(\omega) \mu_i (s,dx)  \big) dA^c_s  \label{eq:char_nu}\\
\int_D \left(e^{\langle u,\xi \rangle}  - 1\right) \nu(\omega,\{t\},d\xi) &= \left(\exp\Big( \gamma_0(t,u) + \sum_{i=1}^d \langle X^i_{t-}(\omega), \bar \gamma_i (t,u) \rangle \Big) - 1\right).\label{eq:char_nu_dis}
\end{align}
\end{subequations}
Moreover, for all $(T,u) \in (0,\infty) \times \cU$, the functions $\phi$ and $\psi$ are absolutely continuous w.r.t $A$ and solve the following generalized measure Riccati equations: their continuous parts satisfy
\begin{align}
\frac{d \phi^c_t(T,u)}{dA^c_t} &= - F(t,\psi_{t}(T,u)), 
\label{ric1}\\
\frac{d \psi^c_t(T,u)}{dA^c_t} &= - R(t,\psi_{t}(T,u)), 
\label{ric2}
\end{align}
$dA^c$-a.e., where
\begin{equation}\label{eq:FR}
\begin{split}
F(s,u) &= \scal{\beta_0(s)}{u} + \half \scal{u}{\alpha_0(s) u} +   \int_D \Big( e^{\scal{x}{u}} -1 - \scal{h(x)}{u}\Big) \mu_0(s,dx)
\\
R_i(s,u) &= \scal{\beta_i(s)}{u} + \half \scal{u}{\alpha_i(s) u} + \int_D \Big( e^{\scal{x}{u}} -1 - \scal{h(x)}{u}\Big) \mu_i(s,dx),
\end{split}
\end{equation}
while their jumps are given by
\begin{align}\label{eq:Deltaphipsi}
\begin{split}
\Delta \phi_{t}(T,u) &= - \gamma_0(t,\psi_t(T,u))   \\
\Delta \psi_{t}(T,u) &= - \bar \gamma(t,\psi_t(T,u)),
\end{split}
\end{align}
and their terminal conditions are
\begin{equation}
\phi_T\brac{T,u}=0 \quad\text{and} \quad\psi_T\brac{T,u}=u.\label{eq:initial-cond}
\end{equation}
\end{theorem}

\begin{remark} \label{rem:parameter_Set}
Note that the parameter set $(A, \gamma,\beta, \alpha, \mu)$ is not uniquely determined: indeed, consider some increasing function $A'$ such that $A \ll A'$ and write $g = \frac{dA}{dA'}$ for the Radon-Nikodym density of $A$ with respect to $A'$. It is easy to see that all statements of the theorem remain true for the alternative parameter set $(A', \gamma,g\beta, g \alpha, g \mu)$.
\end{remark}
\begin{remark}
We expect that Theorem~\ref{prop1} can be extended to affine semimartingales with explosion or killing, by adding a `fourth characteristic' (cf. \cite{schnurr2017fourth} and also \cite{cheridito2005equivalent}), which possesses an affine decomposition similar to \eqref{char:affine}. The rigorous formulation of the corresponding results will not be pursued here, and is left for future research.
\end{remark}

The distribution of the jumps of the affine semimartingale occurring at fixed times $t$ can directly be characterized as follows.

\begin{lemma}\label{lem:affinejumps}
Let $X$ be a quasi-regular affine semimartingale satisfying the support condition~\ref{def:full_support} and with characteristics $(B,C,\nu)$. 
\begin{enumerate}[(i)]
\item For any $(t,u) \in (0,\infty) \times \cU$, 
\begin{equation}\label{eq:jump2}
	\int_{D} \left(e^{\scal u \xi} - 1\right) \nu(\omega; \{t\},d\xi) 
		= \exp\Big( - \Delta \phi_t(t,u) - \scal {\Delta \psi_t(t,u)} {X_{t-}}  \Big) - 1.
\end{equation}
\item Set
\begin{equation}\label{eq:J}
\begin{split}
J^\nu &:= \set{t > 0: \PP(\nu(\omega,\set{t},D) >0) > 0}\\
J^{\phi,\psi} &:= \set{t > 0: \exists\,u \in \cU \text{ such that } \Delta \phi_t(t,u) \neq 0 \text{ or } \Delta \psi_t(t,u) \neq 0}.
\end{split}
\end{equation}
Then $J^\nu = J^{\phi,\psi}$.
\item Set $\gamma_0(t,u) = -\Delta\phi_t(t,u)$ and $\bar\gamma(t,u) = - \Delta \psi_t(t,u)$. Then \eqref{eq:char_nu_dis} and \eqref{eq:Deltaphipsi} hold true and $\gamma = (\gamma_0, \bar \gamma)$ is a good parameter in the sense of Definition~\ref{def:good-parameter} whenever $J^\nu \subset J^A$. 
\end{enumerate}
\end{lemma}
\begin{proof}
By definition, $\nu(\{t\},d\xi)$ is the dual predictable projection of $\delta_{\Delta X_t}(d\xi)$ such that (by Proposition II 1.17 in \cite{JacodShiryaev})
\begin{align*}
 \int_{D} \left(e^{\scal u \xi} - 1\right) \nu(\omega; \{t\},d\xi) 
&= \cond{\left(e^{ \scal u {\Delta X_{t}}}  - 1\right)}{\cF_{t-}}.
\end{align*}
Combining with \eqref{eq:jump}, claim (i) follows. For (ii), let $t \in J^\nu$. Then, there exists an $u \in \cU$, such that the left hand side of \eqref{eq:jump2} is non-zero. Thus also the right hand side is non-zero and we conclude that either $\Delta \phi_t(t,u) \neq 0$ or $\Delta \psi_t(t,u) \neq 0$. It follows that $t \in J^{\phi,\psi}$ and hence that $J^\nu \subseteq J^{\phi,\psi}$. For the other direction let $t \in J^{\phi,\psi}$ and choose an $u \in \cU$ such that $\Delta \phi_t(t,u) \neq 0$ or $\Delta \psi_t(t,u) \neq 0$. Together with Condition~\ref{def:full_support} on $X$ we conclude that the right hand side of \eqref{eq:jump2} is non-zero with strictly positive probability. The same must hold for the left hand side and we conclude that $t \in J^\nu$ and hence that $J^\nu = J^{\phi,\psi}$.
For (iii) note that $\gamma$ has been defined in such a way that \eqref{eq:jump2} becomes \eqref{eq:char_nu_dis}. The jump equations \eqref{eq:Deltaphipsi} are a direct consequence of \eqref{eq:computeDeltaphipsi}. If $J^\nu \subset J^A$, then $\gamma(t,u) = 0$ whenever $t \not \in J^A$ and it follows that $\gamma$ is a good parameter.
\end{proof}

We now focus on the continuous parts of the semimartingale characteristics, and make the following definition: For any affine semimartingale $X$ with characteristics $(B,C,\nu)$ and for $(T,u) \in \Rplus \times \cU$ we define a complex-valued random measure on $[0,T]$ by
\begin{align}\label{eq:G}
G(dt,\omega,T,u) &:= \scal{\psi_t}{dB^c_t(\omega)} + \frac{1}{2}\scal{\psi_t}{dC_t(\omega)\psi_t} + \\
&+ \int_D \left(e^{\langle \psi_t, \xi \rangle} - 1 -\langle \psi_t,h(\xi) \rangle\right) \nu^c(\omega,dt,d\xi),\notag
\end{align}
where we write $\psi_t := \psi_t(T,u)$ for short.

\begin{lemma}\label{lem:characteristics}
Let $X$ be a quasi-regular affine semimartingale with a good version of its characteristics $(B,C,\nu)$, let $(T,u) \in (0,\infty) \times \cU$ and let $G(dt,\omega,T,u)$ be the complex-valued random measure defined in \eqref{eq:G}. It holds that 
\begin{equation}\label{eq:measure_identity}
G(dt;\omega,T,u) + d\phi^c_{t}(T,u) + \scal{X_t(\omega)}{d\psi^c_{t}(T,u)} = 0, \qquad\PP-a.s,
\end{equation}
as identity between measures on $[0,T]$.
\end{lemma}

\begin{proof}For $(T,u) \in (0,\infty) \times \cU$ consider the process
\[
M^{u,T}_{t}:=\E\Big[e^{ \langle u , X_{T} \rangle}\big|\cF_{t}\Big]
	=\exp\left(\phi_t\left(T,u\right)+\scal{\psi_t\left(T,u\right)}{X_{t}}\right)\quad t \in [0,T),
\]
which is a c\`adl\`ag martingale with the terminal value $M_{T}^{u,T} = \exp\left(\scal{u}{X_{T}}\right)$. To alleviate notation we consider $(T,u)$ fixed and write 
\[
M_{t} = M^{u,T}_{t} = \exp\left(\phi_{t}+\scal{\psi_{t}}{X_{t}}\right),
\]
with $\phi_{t} := \phi_t(T,u)$ and $\psi(t) := \psi_t(T,u)$.
Applying the It\^o-formula for semimartingales (cf. \cite[Prop.II.2.42]{JacodShiryaev}) to $M$ we obtain a decomposition 
\[M_t = L_t + F_t,\]
where $L$ is a local martingale and $F$ is the predictable finite variation process
\begin{align} 
F_t := &\int_0^t M_{s-} \left\{ d\phi^c_s + \scal{X_{s-}}{d\psi^c_s} + \scal{\psi_{s-}}{dB_s}  +\frac{1}{2}\scal{\psi_{s-}}{dC_s \psi_{s-}}  \right. \label{eq:prc-char}
\\
&+  \left.\int_D \left(e^{\Delta \phi_s + \langle \psi_s, X_{s-}+\xi\rangle - \langle \psi_{s-},X_{s-}\rangle } - 1  - \langle \psi_{s-},h(\xi) \rangle \right) \nu(\omega,ds,d\xi) \right\}.
\notag
\end{align}
The jump part $\Delta F$ vanishes due to Lemma \ref{lem:affinejumps} and \eqref{eq:B_jumps}, and we are left with the continuous part
\begin{align*}
F_t = F_t^c = &\int_0^t M_{s-} \left\{ d\phi^c_s + \scal{X_{s-}}{d\psi^c_s} + \scal{\psi_{s-}}{dB^c_s}  +\frac{1}{2}\scal{\psi_{s-}}{dC_s \psi_{s-}}  \right. \label{eq:prc-char}
\\
&+  \left.\int_D \left(e^{ \langle \psi_{s-},\xi\rangle } - 1  - \langle \psi_{s-},h(\xi) \rangle \right) \nu^c(\omega,ds,d\xi) \right\}.
\notag
\end{align*}
Recall that $M$ is a martingale, and hence $M \equiv L$ and $F \equiv 0$ on $[0,T]$, $\PP$-a.s. With \eqref{eq:G},  $F$ can be rewritten as
\[F_t = \int_0^t M_{s-} \left\{ d\phi^c_{s} + \scal{X_{s-}}{d\psi^c_{s}} + G(ds;\omega,T,u) \right\}.\]
Since none of the measures appearing above charges points, the left limits $X_{s-}, \psi_{s-}$ can be substituted by right limits $X_s, \psi_s$. Moreover, $M_{s-}$ is nonzero everywhere and \eqref{eq:measure_identity} follows.
\end{proof}

In order to make efficient use of the support condition~\ref{def:full_support}, we introduce the following convention: Given an affine semimartingale $X$, a tuple $\mathbf{X} = (X^0, \dotsc, X^d)$ represents $d+1$ \emph{stochastically independent} copies of $X$. Formally, the tuple $\mathbf{X}$ can be realized 
on the product space $(\Omega^{(d+1)}, \cF^{\otimes (d+1)}, (\cF_t^{\otimes (d+1)})_{t \geq 0})$ equipped with the associated product measure. 
Moreover, for any points $\xi_0, \dotsc, \xi_d$ in $\R^d$, we define the $(d+1) \times (d+1)$-matrix
\begin{equation}\label{eq:H_def}
H(\xi_0, \dotsc, \xi_n) := \begin{pmatrix} 1 & \xi_0^\top \\ \vdots & \vdots \\ 1 & \xi_n^\top \end{pmatrix}.
\end{equation}
The matrix-valued process $\Theta_t$ is formed by inserting $\mathbf{X} = (X^0, \dotsc, X^d)$ into $H$, i.e. we set 
\begin{equation}\label{eq:theta_def}
\Theta_t(\omega) = H(X^0, \dotsc, X^d) 
= \begin{pmatrix} 1 & X_t^0(\omega)^\top \\  \vdots & \vdots \\ 1 & X_t^d(\omega)^\top \end{pmatrix}.
\end{equation}

\begin{lemma}\label{lem:inversion}
Let $s > 0$ and let $X$ be an affine semimartingale satisfying the support condition~\ref{def:full_support}. Then there exists $\epsilon > 0$ and a set $E \in \cF_s$with $\P\brac{E}>0$, such that the matrices $\Theta_t(\omega)$ and $\Theta_{t-}(\omega)$; are regular for all $(t,\omega) \in (s,s+\epsilon) \times E$. 
\end{lemma}

\begin{proof}
Define the first hitting time
\[\tau :=  \inf \set{t > s: \Theta_t \text{ singular, or } \Theta_{t-} \text{ singular}}.\]
Since the set of singular matrices is a closed subset of the vector space of $\RR^{(d+1) \times (d+1)}$-matrices, $\tau$ is a stopping time, cf. \cite[Thm.~1.4]{Protter}. Moreover, by monotone convergence, we have
\[\lim_{n \to \infty} \P\Big(\Theta_t \text{ and } \Theta_{t-} \text{ regular for all $t \in (s,s+1/n)$}\Big) = \lim_{n \to \infty}\P (\tau \ge s+1/n) = \P(\tau > s).\] 
If we can show that $\P(\tau > s) > 0$, then the claim follows by choosing $N$ large enough and setting $\epsilon = 1/N$ and $E = \set{\tau \ge s+1/N}$. But by right-continuity of $X$, the set $\set{\omega:\tau\brac{\omega} > s}$ is equal to $\set{\omega:\Theta_s(\omega) \text{ is regular}}$ and it remains to show that $\Theta_s$ is regular with strictly positive probability. By Condition~\ref{def:full_support} it holds that $\conv(\supp(X_s)) = D$ and we can find $d+1$ convex independent points\footnote{A set of points is called \emph{convex independent} if none of them can be expressed as a convex combination of the remaining points.} $\xi^0, \dotsc, \xi^d$ in $\supp(X_s)$. Recalling the definition of $H$ in \eqref{eq:H_def}, it follows that $H(\xi^0, \dotsc, \xi^d)$ is regular. Since the set of regular matrices is open we find $\delta > 0$ such that even $H(y_0, \dotsc, y_d)$ is regular for all $y_i \in U_{\delta}(\xi_i), i \in \set{0, \dotsc, d}$, where $U_{\delta}(\xi_i)$ is the open ball of radius $\delta$ centered at $\xi_i$. Now, by independence of $X^0, \dots, X^d$, it follows that 
 \begin{align*}
 \PP \left(\Theta_{s} \; \text{is regular}\right) & \ge \PP \Big(X_s^i \in U_\delta(\xi_i) \quad \forall\, i \in \set{0, \dotsc, d}\Big)\\
 &= \prod_{i=0}^d \PP\left(X_s \in U_\delta(\xi_i)\right). 
 \end{align*}
Since for each $i \in \set{0, \dotsc, d}$ the intersection of $U_\delta(\xi_i)$ with the support of $X_s$ is non-empty, all probabilities are strictly positive, and the proof is complete. 
 \end{proof}

Similar to the $\RR^{(d+1) \times (d+1)}$-valued process process $(\Theta_t)_{t \ge 0}$ defined in \eqref{eq:theta_def}, we define $d+1$ independent copies of  the complex-valued random measure $G(dt,\omega,T,u)$  from equation \eqref{eq:G} and denote them by $G_0,\dots,G_d$, respectively. With this notation and for any $(T,u) \in \Rplus \times \cU$, the $d+1$ corresponding equations \eqref{eq:measure_identity} can be written in matrix-vector form as 
\begin{equation}\label{eq:phi_to_G}
\Theta_{t}(\omega) \cdot \begin{pmatrix} d\phi^c_t(T,u) \\ d\psi^{c,1}_{t}(T,u) \\ \vdots \\ d\psi^{c,d}_{t}(T,u) \end{pmatrix}
 =  - \begin{pmatrix} G_{0}(dt;\omega,T,u) \\ \vdots \\ G_{d}(dt;\omega,T,u)  \end{pmatrix}\end{equation}
which holds $\PP$-a.s. as an identity between complex-valued measures on $[0,T]$. The next Lemma gives a 'local' version of the continuous part of Theorem~\ref{prop1}.

\begin{lemma}\label{lem:local}
Let $X$ be a quasi-regular affine semimartingale satisfying the support condition~\ref{def:full_support} and let $\tau \in (0,\infty)$ be a deterministic timepoint. Then there exists an interval $I_\tau := (\tau, \tau+ \epsilon)$, where $\epsilon = \epsilon(\tau) > 0$, and good parameters $(A^c, \beta, \alpha, \mu)$ on $I_\tau$. With respect to these parameters, and with $F$ and $R$ as in \eqref{eq:FR}, the measure Riccati equations \eqref{ric1} and \eqref{ric2} hold true for each $(T,u) \in \Rplus \times \cU$ and $t \in I_\tau \cap [0,T]$.
\end{lemma}

\begin{remark}
We emphasize that in this lemma the parameters $(A^c, \beta, \alpha, \mu)$ as well as the functions $F$ and $R$ may depend on $\tau$. 
\end{remark}

For a semimartingale $X$ there exists a c\`adl\`ag,  increasing, predictable, $\Rplus$-valued  process $\cA$ starting in $0$ and with continuous part $\cA^c$, such that the semimartingale characteristics of $X$ can be `disintegrated' with respect to $\cA$. For the continuous parts $(B^c,C,\nu^c)$ of the characteristics, this implies the representation
\begin{align}
B^c_t &= \int_0^t b_s d\cA^c_s\notag\\
C_t &= \int_0^t c_s d\cA^c_s \label{eq:char_disintegration}\\
\nu^c(\omega,dt,dx) &= K_{\omega,t}(dx) d\cA^c_t(\omega) \notag,
\end{align}
where $b$ and $c$ are predictable processes and $K_{\omega,t}(dx)$ a transition kernel from $\Omega \times \Rplus$, endowed with the predictable $\sigma$-algebra, to ($\RR^d, \cB(\RR^d))$; see \cite[Prop.~II.2.9]{JacodShiryaev} for further details. 

\begin{proof}
Let $X^0,\dots,X^d$ be $d+1$ stochastically independent copies of $X$. Denote the semimartingale characteristics of  $X^{i}$ by $(B^i,C^i,\nu^i)$ and define $G_{i}(\omega;t,T,u)$ as in \eqref{eq:G}, $i=0,\dots,d$. The semimartingale characteristics $(B^i,C^i,\nu^i)$ can be disintegrated as in \eqref{eq:char_disintegration}. Since we consider only a finite collection of semimartingales, we may assume that the process $\cA^c_s(\omega)$ is the same for each $X^{i}$. 

By Lemma~\ref{lem:inversion}, there exists an interval $I_\tau=(\tau, \tau + \epsilon)$, $\epsilon>0$, and a set $E \in \cF$ with $\PP(E) > 0$ and such that $\Theta_{t}(\omega)$ is invertible for all $(t,\omega) \in I_\tau \times E$. Multiplying \eqref{eq:phi_to_G} from the left with the inverse of this matrix yields
\begin{equation}\label{eq:theta_inverse}
\begin{pmatrix}
d\phi^c_{t}(T,u) \\ 
d\psi^{c,1}_{t}(T,u) \\
\vdots \\
d\psi^{c,d}_{t}(T,u)
\end{pmatrix}
 =  - \Theta_{t}(\omega)^{-1} \cdot 
\begin{pmatrix} 
G_{0}(dt;\omega,T,u) \\
\vdots \\
G_{d}(dt;\omega,T,u)
\end{pmatrix},
\end{equation}
as an identity between complex-valued measures on $I_\tau$ for all $\omega \in E$. 
Since $\PP(E) > 0$, we can choose some particular $\omega_* \in E$ where \eqref{eq:theta_inverse} holds. Setting 
\[A^c_t := \cA^c_t(\omega_*), \qquad t \in I_\tau\]
 we observe that $G_i(dt;\omega_*,T,u) \ll dA^c_t$ for each $i \in \set{0, \dotsc, d}$ and conclude that also the left hand side of \eqref{eq:theta_inverse} is absolutely continuous with respect to $A^c$ on $I_\tau$. Denote by $(b^i,c^i,K^i)$ the disintegrated semi-martingale characteristics of $X^i$, as in \eqref{eq:char_disintegration}. Note that the random measures $G_i(dt;\omega,T,u)$ depend linearly on $(b^i, c^i, K^i)$, which in light of \eqref{eq:theta_inverse}  suggests to apply the linear transformation $\Theta_{t}(\omega)^{-1}$ directly to the disintegrated semimartingale characteristics. Evaluating at $\omega_*$, we hence define the \emph{deterministic} functions $(\beta^i, \alpha^i, \mu^i)_{i \in \set{0, \dotsc, d}}$ on $I_\tau$ by setting
\begin{align*}
\left(\beta^0 , \beta^1 , \dotsc , \beta^d \right)_t^\top &:= \Theta_{t-}(\omega_*)^{-1} \cdot \left(b^0 , b^1 , \dotsc , b^d \right)_t^\top (\omega_*)\\
\left(\alpha^0_{kl} , \alpha^1_{kl} , \dotsc , \alpha^d_{kl} \right)_t^\top
&:= \Theta_{t-}(\omega_*)^{-1} \cdot \left(c^0_{kl}, c^1_{kl} , \dotsc , c^d_{kl}  \right)_t^\top(\omega_*), \quad k,l \in \set{1, \dotsc, d}\\
\left( \mu^0 , \mu^1 , \dotsc , \mu^d \right)_t^\top
&:= \Theta_{t-}(\omega_*)^{-1} \cdot \left(K^0 , K^1 , \dotsc , K^d  \right)_t^\top(\omega_*).
\end{align*}
Using these parameters, the functions $F,R$ can be defined on $I_\tau$ as in \eqref{eq:FR}. In combination with \eqref{eq:theta_inverse} it follows that 
\begin{equation}\label{eq:psi_FR}
\begin{pmatrix}
d\phi^c_{t}(T,u) \\ 
d\psi^{c,1}_{t}(T,u) \\ 
\vdots \\ 
d\psi^{c,d}_{t}(T,u) 
\end{pmatrix}
=  - \Theta_{t}(\omega_*)^{-1} \cdot 
\begin{pmatrix} 
G_0(dt;\omega_*,T,u) \\ 
\vdots \\ 
G_d(dt;\omega_*,T,u)  
\end{pmatrix}
= - 
\begin{pmatrix} 
F(t,\psi_{t}( T,u)) \\ 
R^1(t,\psi_{t}( T,u)) \\ 
\vdots \\ 
R^d(t,\psi_{t}( T,u))  
\end{pmatrix} dA^c_t
\end{equation}
for $t \in I_\tau \cap [0,T]$, which yields validity of the Riccati equations \eqref{ric1} and \eqref{ric2} on $I_\tau$.
\end{proof}

\begin{proof}[Proof of Thm.~\ref{prop1}]We consider first the continuous parts of the Riccati equations, and thereafter treat their jumps. Applying Lemma~\ref{lem:local} to each $\tau \in (0,\infty)$ we obtain a family of intervals $I_\tau$, each with non-empty interior $I_\tau^\circ$, such that $(I_\tau^\circ)_{\tau \in (0,\infty)}$ is an open cover of the positive half-line $(0,\infty)$. Since $\Rplus$ can be exhausted by compact sets such a cover has a countable subcover $\cS$. To each interval $I \in \cS$, Lemma~\ref{lem:local} associates good parameters $(A^{c,I}, \beta^{I}, \alpha^{I}, \nu^{I})$. By countability of $\cS$ there exists a continuous common dominating function $A^c: \Rplus \to \Rplus$ such that $A^{c,I} \ll A^c$ for all $I \in \cS$. As discussed in Remark~\ref{rem:parameter_Set}, passing from $A^{c,I}$ to $A^c$ has merely the effect of multiplying all parameters with the Radon-Nikodym derivative $\frac{dA^{c,I}}{dA^c}$. Hence, we may assume without loss of generality that $A^{c,I} = A^c$ for each $I \in\cS$. 

Let now $I$ and $\tilde I$ be two intervals with non-empty intersection, taken from the countable subcover $\cS$. Denote by $(A^c, \beta, \alpha, \mu)$ and $(A^c, \tilde \beta, \tilde \alpha, \tilde \mu)$ the respective parameter sets obtained for these intervals by application of Lemma~\ref{lem:local} and by $(F,R)$ and $(\tilde F, \tilde R)$ the corresponding functions defined by \eqref{eq:FR}. We say that these two parameter sets are \emph{compatible} if they agree (up to a $dA^c_t$-nullset) on the intersection $I \cap \tilde I$. Once we have shown compatibility for arbitrary intervals $I$ and $\tilde I$ it is clear that we can find a single good parameter set 
$(A, \beta, \alpha, \mu)$, defined on the whole real half-line $\Rplus$, such that the Riccati equations \eqref{ric1} and \eqref{ric2} hold true. To condense notation, we introduce the vectors
\[d\Psi^c_{t}(T,u) := \begin{pmatrix} d\phi^c_{t}(T,u) \\ d\psi^{c,1}_{t}(T,u) \\ \vdots \\ d\psi^{c,d}_{t}(T,u) \end{pmatrix}, \quad \cR(t,u) := 
\begin{pmatrix} F(t,u) \\ R^1(t,u) \\ \vdots \\ R^d(t,u)  \end{pmatrix}, \quad \quad \tilde \cR(t,u) := 
\begin{pmatrix} \tilde F(t,u) \\ \tilde R^1(t,u) \\ \vdots \\ \tilde R^d(t,u)  \end{pmatrix}.\]
Applying equation \eqref{eq:psi_FR} once on the interval $I$ and once on $\tilde I$ yields
\begin{equation}\label{eq:RR}
\cR(t,\psi_{t}(T,u)) dA^c_t = d\Psi^c_{t}(T,u) = \tilde \cR(t,\psi_{t}(T,u)) dA^c_t, \quad t \in I \cap \tilde I \cap [0,T].
\end{equation}
Let now $\cT \times \cE$ be a countable dense subset of $\Rplus \times \cU$.  Taking the union over the countable set $\cT \times \cE$ we obtain from \eqref{eq:RR} that
\begin{equation}\label{eq:RR2}
\cR(t,\psi_{t}(T,u)) = \tilde \cR(t,\psi_{t}(T,u)) \quad \text{for all } (T,u) \in \cT \times \cE \text{ and } t \in (I \cap \tilde I \cap [0,T]) \setminus N,
\end{equation}
where $N$ is a $dA^c_t$-nullset, independent of $(T,u)$. 

The next step is to `evaluate' \eqref{eq:RR2} at $T = t$ and to use that $\psi_{t}( t, u) = u$ by taking limits in the countable set $\cT$. Observe  that as functions of L\'evy-Khintchine-form (cf. \eqref{eq:FR}) both $F$ and $R$ are continuous in $u$. 
By denseness of $\cT$ in $\Rplus$ we can find a sequence $(T_n) \subseteq \cT$ such that $T_n \downarrow t$ as $n \to \infty$. 

Together with the right-continuity of $\psi_t(T,u)$ in $T$ this yields
\begin{equation}\label{eq:RR4}
\cR(t,u) = \lim_{n \to \infty } \cR(t,\psi_{t}(T_n,u))  = \lim_{n \to \infty }\tilde \cR(t,\psi_{t}(T_n,u)) = \tilde \cR(t,u), 
\end{equation}
for all $u \in \cE$.  Using continuity of $F$ and $R$  in $u$, Equation \eqref{eq:RR4} can be extended from the dense subset $\cE$ to all of $\cU$. It is well-known that a function of L\'evy-Khintchine-form determines its parameter triplet uniquely, cf. \cite[Thm.~8.1]{Sato}. Hence, we may conclude that 
\[\beta^i_t = \tilde \beta^i_t, \quad \alpha^i_t = \tilde \alpha^i_t, \quad \mu^i_t = \tilde \mu^i_t,\]
for each $i \in \set{0, \dotsc, d}$ and $t \in I \cap \tilde I$ with exception of the $dA^c_t$-nullset $N$. This is the desired compatibility property and shows the existence of good parameters $(A^c, \beta, \alpha, \nu)$. 

We now turn to the continuous parts of the semimartingale characteristics $(B,C,\nu)$ and show \eqref{eq:char_B}, \eqref{eq:char_C} and \eqref{eq:char_nu}. To this end, fix $(T,u) \in \Rplus \times \cU$ and let $(b,c,K)$ be the continuous semimartingale characteristics of $X$, disintegrated with respect to the increasing predictable process $\cA^c_t(\omega)$, as in \eqref{eq:char_disintegration}. For each $\omega \in \Omega$, write 
\[\cA^c_t(\omega) = \int_0^t a_s(\omega) dA^c_t + \cS_t(\omega)\]
for the Lebesgue decomposition of $\cA^c_t(\omega)$  with respect to $A^c_t$.\footnote{Note that our argument does not require measurability of $\omega \mapsto a_s(\omega)$ or $\omega \mapsto \cS_t(\omega)$.} Furthermore, define
\begin{align}
g(\omega,t,T,u) &:= \scal{\psi_{t}}{b_t(\omega)} + \frac{1}{2}\scal{\psi_{t}}{c_t(\omega)\psi_t} + \\
&+ \int_D \left(e^{\scal{\psi_{t}}{\xi}} - 1 - \scal{\psi_{t}}{h(\xi)} \right) K_t(\omega,d\xi),\notag
\end{align}
which can be considered as the disintegrated analogue of \eqref{eq:G}. Combining \eqref{eq:phi_to_G} with the Riccati equations, we obtain that
\begin{equation}
\Theta_{t}(\omega;x) \cdot \cR(t, \psi_{t}(T,u)) dA^c_t = g(\omega,t,u,T) a_t(\omega) dA^c_t + g(\omega,t,u,t) d\cS_t(\omega)
\end{equation}
for all $(T,u) \in \Rplus \times \cU$ and $t \in [0,T]$. By the uniqueness of the Lebesgue decomposition we conclude that
\begin{equation}\label{eq:small_g}
\begin{cases} a_t(\omega)g(\omega,t,T,u)  = \Theta_{t}(\omega) \cdot \cR(t, \psi_{t}(T,u)), &\qquad dA^c_t-a.e\\
\phantom{a_t(\omega)}g(\omega,t,T,u) = 0, &\qquad d\cS_t(\omega)- a.e.
\end{cases} 
\end{equation}
As in the first part of the proof, we consider a countable dense subset $\cT \times \cE$ of $\Rplus \times \cU$. Taking the union over all $(T,u)$ in $\cT \times \cE$ and repeating the density arguments of \eqref{eq:RR4} we find an $dA^c_t$-nullset $N_1$ and a $d\cS_t(\omega)$-nullset $N_2$, such that
\begin{equation}
\begin{cases} a_t(\omega)g(\omega,t,t,u)  = \Theta_{t}(\omega) \cdot \cR(t, u), &\quad \text{for all } t \in \Rplus \setminus N_1, u \in \cE\\
\phantom{a_t(\omega)}g(\omega,t,t,u) = 0, &\quad \text{for all } t \in \Rplus \setminus N_2, u \in \cE.
\end{cases} 
\end{equation}
As functions of $u$, both sides are of L\'evy-Khintchine-form. In addition, $\cE$ is dense in $\cU$, which allows us to conclude from the first equation that
\begin{align*}
a_t(\omega)b_t(\omega)  &= \Theta_{t}(\omega) \cdot (\beta^0_t, \dotsc, \beta^d_t)\\
a_t(\omega)c_t(\omega)  &= \Theta_{t}(\omega) \cdot (\alpha^0_t, \dotsc, \alpha^d_t)\\
a_t(\omega) K_t(\omega) &= \Theta_{t}(\omega) \cdot (\mu^{c,0}_t, \dotsc, \mu^{c,d}_t)
\end{align*}
for all $t \in \Rplus \setminus N_1$ and from the second equation that
\[b_t(\omega)= 0, \quad c_t(\omega) = 0, \quad K_t(\omega) = 0, \qquad d\cS_t(\omega)-a.e.\]
Integrating with respect to $\cA^c_t(\omega)$ and adding up yields\eqref{char:affine}.\\

To conclude the proof, we finally turn to the discontinuous part. Note that Lemma~\ref{lem:affinejumps} already provides us with parameters $\gamma$, a set $J^\nu$ and the validity of \eqref{eq:char_nu} and \eqref{eq:Deltaphipsi}. Taking the continuous increasing function $A^c$ from the first part of the proof and inserting jumps of strictly positive hight at each time $t \in J^\nu$ we obtain an increasing function $A$ with continuous part $A^c$ and jump set $J^A = J^\nu$. Note that the heights of the jumps are arbitrary; for example the values of the summable series $(2^{-n})_{n \in \NN}$ can be taken. Together, $(A, \gamma, \alpha, \beta, \mu)$ is now a good parameter set in the sense of Definition~\ref{def:good-parameter} and all parts of Theorem~\ref{prop1} have been shown.
\end{proof}

\section{Affine Markov processes and infinite divisibility}

Let $X$ be a Markov process in $D$ (possibly non-conservative) with transition kernels $p_{s,t}(x,B)$, defined for all $0 \le s \le t$, $x \in D$ and $B \in \cB(D)$. The following definition is analogous to \cite[Def.~2.1]{DuffieFilipovicSchachermayer}. 
\begin{definition}\label{def:affine_markov}
A Markov process $X$ in $D$ is called \emph{affine Markov process}, if there exist $\C$- and $\C^d$-valued functions $\phi, \psi$, such that the transition kernels of $X$ satisfy
\begin{equation}\label{eq:Markov}
\int_D e^{\scal{u}{\xi}} p_{s,t}\brac{x,d\xi}=e^{\phi_s\brac{t,u}+\scal{\psi_s\brac{t,u}}{x}},
\end{equation}
for all $0 \le s \le t$, $(x,u) \in D \times \cU$. 
\end{definition}
Under mild conditions, affine semimartingales are also affine Markov processes. First, note that to every affine semimartingale we can associate transition kernels $p_{s,t}(x,B)$, defined for all $0 \le s \le t$, $B \in \cB(D)$ and $x \in \supp(X_s)$, by considering the regular conditional distributions
\begin{equation}\label{eq:reg_cond}
\PP\left(\left.X_t \in B \right| X_s\right) = p_{s,t}(X_s,B).
\end{equation}
By \eqref{cond:affine}, the kernels will satisfy \eqref{eq:Markov} for all $x \in \supp(X_s)$ and the semi-flow-equations \eqref{eq:flow} provide the Chapman-Kolmogorov equations for the kernels $p_{s,t}(x,.)$. It remains to show that the family of transition kernels and the validity of \eqref{eq:Markov} can be extended from $\supp(X_s)$ to $D$. Apart from the trivial condition $\supp(X_s) = D$ for all $s > 0$, we can give the following sufficient condition:

\begin{definition}An affine semimartingale $X$ is called infinitely divisible, if the regular conditional distributions $p_{s,t}(X_s,.)$ are infinitely divisible probability measures on $D$, $\PP$-a.s. for any $0 \le s \le t$. 
\end{definition}

\begin{lemma}\label{lem:markov}
Let $X$ be a quasi-regular affine semimartingale satisfying the support condition~\ref{def:full_support}. Suppose that 
\begin{enumerate}[(i)]
\item $\supp(X_t) = D$ for all $t > 0$, or
\item $X$ is infinitely divisible.
\end{enumerate}
Then $X$ can be realized as a conservative affine Markov process with state space $D$.
\end{lemma}
\begin{proof}
It suffices to show that the right side of \eqref{eq:Markov} is the Fourier transform of a probability measure on $D$ for all $x \in D$ and $0 < s \le t$. Indeed, if the family $\brac{p_{s,t}}_{0 \le s\leq t}$ satisfies \eqref{eq:Markov}, the semiflow equations \eqref{eq:flow} ensure that it satisfies the Chapman-Kolmogorov equations. By the Kolmogorov existence theorem (see, e.g., \cite[Theorem 8.4]{Kallenberg2002}), this guarantees the existence of a unique Markov process with transition kernels $\brac{p_{s,t}}_{0 \le s\leq t}$. 
Let $p_{s,t}(x, .)$ be the transition kernels of the semimartingale $X$, defined by \eqref{eq:reg_cond}. Note that by the affine property \eqref{cond:affine}, these kernels satisfy \eqref{eq:Markov} for all $x \in \supp(X_s)$, and it remains to extend the identity to all $x \in D$. In case (i), this is trivial for $s > 0$, since $\supp(X_s)= D$. In case (ii), by infinite divisibility, there exists, for any $\lambda \in (0,1)$, a probability kernel $p^{(\lambda)}_{s,t}(x, .)$, such that 
\begin{equation}
\int_D e^{\scal{u}{\xi}} p^{(\lambda)}_{s,t}\brac{x,d\xi}=e^{\lambda \phi_s\brac{t,u}+\scal{\psi_s\brac{t,u}}{\lambda x}}.
\end{equation}
Fix $x, y \in \supp(X_s)$, $\lambda \in (0,1)$ and let $z = \lambda x + (1-\lambda) y$ be a convex midpoint of $x$ and $y$. At $z$ we define $p_{s,t}(z,.) := p^{\lambda}_{s,t}(x,.) \star p^{(1-\lambda)}_{s,t}(y,.)$, where $\star$ denotes convolution of measures, and obtain
\begin{equation}
\int_D e^{\scal{u}{\xi}} p_{s,t}(z,d\xi)=e^{\phi_s\brac{t,u}+\scal{\psi_s\brac{t,u}}{\lambda x + (1-\lambda)y}} = e^{\phi_s\brac{t,u}+\scal{\psi_s\brac{t,u}}{z}}
\end{equation}
i.e. \eqref{eq:Markov} has been extended to the convex midpoint $z = \lambda x + (1-\lambda) y$ of $x$ and $y$. By Condition~\ref{def:full_support} we have $\conv(\supp(X_s)) = D$ for all $s > 0$, which shows \eqref{eq:Markov}, except at the time-point $s = 0$. In both cases (i) and  (ii) we can finally use the quasi-regularity property of $\phi, \psi$ to immediately extend \eqref{eq:Markov} to $s = 0$ by taking limits from the right.
\end{proof}

It turns out that infinite divisibility has even stronger implications on the structure of affine semimartingales, in particular at the deterministic jump times $J^A$.

\begin{lemma}\label{lem:id_jumps}Let $X$ be an infinitely divisible, quasi-regular affine semimartingale satisfying the support condition~\ref{def:full_support}. Then the conditional distribution of $\Delta X_t$ given $X_{t-}$ is $\PP$-a.s infinitely divisible, for any $t \ge 0$. Moreover, the parameters $\gamma = (\gamma_0, \gamma_1, \dotsc, \gamma_d)$ in Theorem~\ref{prop1} are of the following form: For any $t \in J^A$ and $i \in \set{0, \dotsc, d}$, there exist $\tilde \beta_i(t) \in \RR^d$, $\tilde \alpha_i(t) \in \cS^d$ and a (possibly signed) Borel measure $\tilde \mu_i(t,.)$ on $D \setminus \{0\}$, such that 
\begin{equation}\label{eq:gamma_decomp1}
\gamma_i(t,u) = \scal{\tilde \beta_i(t)}{u} + \frac{1}{2}\scal{u}{\tilde \alpha_i(t) u} + \int_D \left(e^{\scal{x}{u}} - 1 - \scal{h(x)}{u}\right) \tilde \mu_i(t,dx),
\end{equation}
for all $u \in \cU$.
\end{lemma}
\begin{proof}
Using Lemma~\ref{lem:affinejumps} and the quasi-regularity property from Definition~\ref{def:quasi_regular}, we can write
\begin{align*}
E\big[ e^{\scal{u}{X_{t}}} |\cF_{t-} \big] &= \exp\left(-\Delta \phi_t(t,u)  - \scal{\Delta \psi_t(t,u)}{X_{t-}} \right) = \\
&= \lim_{s \uparrow t }\exp\left(\phi_s(t,u) + \scal{\psi_s(t,u)}{X_s} \right) = \lim_{s \uparrow t} \int_D e^{\scal{u}{\xi}} p_{s,t}\brac{X_s,d\xi}.
\end{align*}
Note that the right hand side is the limit of Fourier-Laplace transforms of infinitely divisible measures on $D$. The left hand side is the Fourier-Laplace transform of the distribution of $X_t$, conditionally on $\cF_{t-}$, and we conclude that also this distribution must be infinitely divisible. By Lemma~\ref{lem:affinejumps} $\gamma_0(t,u) = -\Delta \phi_t(t,u)$ and $\gamma_i(t,u) = -\Delta \psi^i_t(t,u)$ for all $i \in \set{1, \dotsc, d}$. The decomposition \eqref{eq:gamma_decomp1} then follows from the L\'evy-Khinchtine formula for infinitely divisible distributions.
\end{proof}

Recall the definition of a \emph{good parameter set} $(A,\gamma,\beta,\alpha,\mu)$ from Definition~\ref{def:good-parameter}, and note that the functions $\beta(t), \alpha(t)$ and $\mu(t,.)$ are only defined up to $A^c$-nullsets. In particular, we can modify $\beta, \alpha, \mu$ at any jump point $t \in J^A$ without affecting the validity of Theorem~\ref{prop1}. In light of the decomposition \eqref{eq:gamma_decomp1} of $\gamma$ this suggests the following definition:
\begin{definition}\label{def:enhanced}
Let $(A,\gamma,\beta,\alpha,\mu)$ be the good parameter set of an quasi-regular infinitely divisible affine semimartingale $X$ satisfying the support condition~\ref{def:full_support}. We \emph{enhance} the functions $\beta, \alpha, \mu$ by setting
\begin{subequations}\label{eq:FR_frak}
\begin{alignat}{2}
\alpha_i(t) &= \tfrac{1}{\Delta A_t} \tilde{\alpha}_i(t) \qquad \beta_i(t) &&= \tfrac{1}{\Delta A_t} \tilde{\beta}_i(t) \\
\mu_i(t,d\xi) &= \tfrac{1}{\Delta A_t} \tilde{\mu}_i(t,d\xi) && \text{for all $t \in J^A$, $i \in \set{0, \dotsc, d}$},
\end{alignat}
\end{subequations}
with $\tilde{\alpha}, \tilde{\beta}, \tilde{\mu}$ as in Lemma~\ref{lem:id_jumps} and refer to $(A,\beta,\alpha,\mu)$ as \emph{enhanced parameter set} of $X$.
\end{definition}
Note that $\gamma$ does no longer appear in the enhanced parameter set, since it was absorbed into the values of $\alpha, \beta, \mu$ at the time-points $t \in J^A$. The enhanced parameters also allow us to combine $F$ with $\gamma$ and $R$ with $\bar \gamma$ by setting
\begin{align*}
\mathfrak{F}\brac{t,u} &:= F(t,u) \ind{t \not \in J^A} +  \frac{1}{\Delta A_t}\gamma_0(t,u) \ind{t \in J^A},\\
\mathfrak{R}\brac{t,u} &:= R(t,u) \ind{t \not \in J^A} + \frac{1}{\Delta A_t} \bar \gamma(t,u) \ind{t \in J^A}.
\end{align*}
Both $\mathfrak{F}$ and $\mathfrak{R}$ are of L\'evy-Khintchine form and the continuous part \eqref{ric1}-\eqref{ric2} and discontinuous part \eqref{eq:Deltaphipsi} of the measure Riccati equations can be unified into the measure differential equations
\begin{align*}
\frac{d \phi_t(T,u)}{dA_t} &= - \mathfrak{F}(t,\psi_{t}(T,u)),\\
\frac{d \psi_t(T,u)}{dA_t} &= - \mathfrak{R}(t,\psi_{t}(T,u)), 
\end{align*}
which, together with the terminal conditions \eqref{eq:initial-cond}, are equivalent to the integral equations
\begin{subequations}\label{eq:unified-riccati}
\begin{align}
\phi_t\brac{T,u}&=\int_{(t,T]} \mathfrak{F}\brac{s,\psi_s\brac{T,u}}dA_s\label{eq:unified-riccati-phi},\\
\psi_t\brac{T,u}&=u+\int_{(t,T]} \mathfrak{R}\brac{s,\psi_s\brac{T,u}}dA_s.
\label{eq:unified-riccati-psi}
\end{align}
\end{subequations}


\section{Existence of affine Markov processes and affine semimartingales}\label{Existence}

In this section we show, under mild assumptions, the existence of affine semimartingales, using affine Markov processes as an intermediate step. While we have made no restriction on the state space $D$ before, we consider throughout this section only the `canonical state space' (cf. \cite{DuffieFilipovicSchachermayer, Filipovic05}) 
\[D=\R^m_{\ge 0}\times\R^n, \qquad m+n=d.\] 
Note that for this state space, $\cU$ takes the form $\cU = \mathbb{C}_{\leq0}^{m}\times i\mathbb{R}^{n}$. In addition we have
\[
\partial\mathcal{U} =  i\mathbb{R}^{d},\quad\mathcal{U}^{o} = \mathbb{C}_{<0}^{m}\times i\mathbb{R}^{n},
\]
as in \cite{DuffieFilipovicSchachermayer}. For notational simplicity we denote $\mathcal{I}=\left\{ 1,\cdots,m\right\} $, $\mathcal{J}=\left\{ m+1,\cdots,d\right\} $, and $\mathcal{I}\setminus i:=\mathcal{I}\setminus\left\{ i\right\} $, $\mathcal{J}\cup i:=\mathcal{J}\cup\left\{ i\right\} $ for any $i$. 
Finally, we  introduce the following short-hand notations: 
\begin{itemize}
\item For two subsets $I,J \subset \{1,\dots,d\}$ we denote by $a_{IJ}$ the submatrix of $a$ with indices in ${I}\times{J},$ i.e $a_{IJ}:=\left(a_{ij}\right)_{i\in{I},\, j\in{I}}$ 
\item $\beta$ denotes the matrix with columns $\beta_0, \beta_1, \dotsc, \beta_d$. We write $\bar \beta$ for $\beta$ with the first column dropped.
\item For any $i,k \in \set{0, \dotsc, d}$ we set $H_{ik}(t) := \int_{D \setminus \set{0}} h_i(\xi) \mu_k(t,d\xi)$ whenever the integral is finite. The other values can be chosen arbitrarily, and the resulting matrix is denoted by $H(t) = (H_{ik}(t))$.
\end{itemize}

Recall from Theorem \ref{prop1} that an affine semimartingale $X$ has a good parameter set $(A,\gamma, \alpha,\beta,\mu)$. To show existence of an affine semimartingale given a good parameter set we also need to take into account the geometry of our state space. In \cite{DuffieFilipovicSchachermayer} this was done by introducing admissibility conditions on the parameters. In the following definition we extend this notion to  our setting.

\begin{definition}
\label{def:admissible}
A good parameter set $(A,\gamma,\alpha,\beta,\mu)$ is called \emph{admissible}, if
\begin{enumerate}[(i)]
\item for $A^c$-almost all $t \in \Rplus$,
\begin{itemize}
\item $\alpha_i\left(t\right)\in\cS_+^{d}$ for all $ i\in \set{0,\cdots,d }$, $\alpha_{0;\mathcal{II}}\left(t\right)=0$, $\alpha_{i;\mathcal{I}\setminus i,\mathcal{I} \setminus i}(t)=0$ for $i\in \cI$, and $\alpha_j(t) = 0$ for $j \in J$,
\item $\beta(t)\in \R^{d\times (d+1)}$ such that $\beta_0 \in D$, $\bar \beta_{\mathcal{IJ}}\left(t\right)=0$ and $\bar \beta_{i(\mathcal{I}\setminus i)}\left(t\right) - H_{i(\mathcal{I}\setminus i)}(t) \in \mathbb{R}^{m-1}_{\geq0}$ for all $i\in \mathcal{I}$
\item $\mu(t)$ is a vector of L\'evy measures with support on $D$ such that $\mu_j\brac{t}=0$ for $j\in \mathcal{J}$ and $\mathcal{M}_{i}\left(t\right)<\infty$ for $i\in \mathcal{I}\cup{0}$, where,
\begin{equation}\label{eq:calM}
\mathcal{M}_{i}\left(t\right):=\int_{D\setminus\set{0}}\left(\left\langle h_{\mathcal{I}\setminus i}\left(\xi\right),1\right\rangle +\left\Vert h_{\mathcal{J}\cup i}\left(\xi\right)\right\Vert ^{2}\right)\mu_{i}\left(t,d\xi\right).
\end{equation}
\end{itemize}
\item for all $t\in J^A$ and all $x\in D$, the function $u \mapsto \exp\brac{\gamma_0\brac{t,u}+\scal{\bar \gamma \brac{t,u} + u}{x}}$ is the Fourier-Laplace transform of a $D$-valued random variable.
\end{enumerate}
If $X$ is infinitely divisible and $(A,\alpha, \beta, \mu)$ its enhanced parameter set (see Definition~\ref{def:enhanced}), then (ii) can be replaced by
\begin{enumerate}
\item[(ii')] for all $t \in J^A$ and $ i\in \set{0,\cdots,d }$,
\begin{itemize}
\item $\alpha_i\left(t\right)\in\cS_+^{d}$, $\alpha_{i;\mathcal{II}}\left(t\right)=0$ for $i \in \cI \cup 0$ and $\alpha_j(t) = 0$ for $j \in \cJ$,
\item $\beta_0(t) \in D$, $\bar \beta_{\mathcal{IJ}}\left(t\right)=0$ and $\bar \beta_{\mathcal{I}\mathcal{I}}(t)  - H_{\mathcal{I}\mathcal{I}}(t) + \id_d \in \Rplus^m$.
\item $\mu_i(t)$ is a L\'evy measure on $D$ with $\int_{D \setminus 0} \left(\left\langle h_{\mathcal{I}}\left(\xi\right),1\right\rangle +\left\Vert h_{\mathcal{J}}\left(\xi\right)\right\Vert ^{2}\right)\mu_{i}\left(t,d\xi\right) < \infty$ for $i \in \cI \cup 0$ and $\mu_j = 0$ for $j \in J$.
\end{itemize}
\end{enumerate}
\end{definition}

Note that a zero element on the diagonal of a semi-definite matrix implies that the whole corresponding row and column is zero; therefore further restrictions on the elements of $\alpha_i$ can be derived from the above conditions.

\begin{proposition}\label{prop:admissibility}
Let $X$ be a quasi-regular affine semimartingale satisfying the support condition~\ref{def:full_support} with good parameter set $(A,\gamma, \alpha,\beta,\mu)$. Suppose that 
\begin{enumerate}[(i)]
\item $\supp(X_t) = D$ for all $t > 0$, or
\item $X$ is infinitely divisible.
\end{enumerate}
Then the parameters $(A,\gamma, \alpha,\beta,\mu)$ are admissible.
\end{proposition}

\begin{proof}
By Lemma \ref{lem:markov}, $X$ can be realized as a (time-inhomogeneous)  Markov process with transition kernels $p_{s,t}\brac{x,d\xi}$, defined for all $0 \le s \le t$ and $x \in D$. Set $f_u(x)=e^{\scal{u}{x}}$ for $u\in \mathcal{U}$. Similar to the proofs of admissibility in \cite{DuffieFilipovicSchachermayer}  we consider the following limit
\begin{align}\label{eq:limit-generator}
G_t f_u (x) &\coloneqq \lim_{h\downarrow0} \frac{\E\brak{f_u\brac{X_{t}}\vert X_{t-h}=x}-e^{\scal{u}{x}}} {A_t-A_{t-h}} = \\
&= \notag \lim_{h\downarrow0} \frac{\exp \left(\phi_{t-h}(t,u) + \scal{\psi_{t-h}(t,u)}{x}\right) - e^{\scal{u}{x}}}{A_t-A_{t-h}}.
\end{align}
For $A^c$-almost all $t \in \Rplus$, there exists a  sequence $\brac{h_n}_{n\in \mathbb{N}}$, decreasing to $0$, along which the limit exists (c.f. the main Theorem in \cite{Daniell1918} or \cite[Theorem 5.8.8]{BogachevMeasure}).
From \eqref{eq:Utcond2}, together with \eqref{ric1} and \eqref{ric2},  we can identify the limit to be
\begin{equation}
\label{eq:cont-gen}
G_t f_u(x)=\brac{F(t,u)+\scal{R(t,u)}{x}}f_u(x).
\end{equation}
For $t\in J^A$, we obtain instead from \eqref{eq:jump} that
\begin{eqnarray}
\label{eq:discont-gen}
G_s f_u(x)&=&\brac{e^{-\Delta \phi_s\brac{s,u}-\scal{\Delta \psi_s\brac{s,u}}{x}}-1}\cdot\frac{1}{\Delta A_s}f_u(x).
\end{eqnarray}
On the other hand we can write the limit in terms of the transition kernels of $X$ as
\begin{align*}
\frac{G_t f_u(x)}{f_u(x)}&=\lim_{n\rightarrow \infty}\frac{1}{A_t-A_{t-h_n}}\brac{\int_D \Brac{f_u\brac{\xi-x}-1}p_{t-h_n,t}\Brac{x,d\xi}}.
\end{align*}
By \eqref{eq:cont-gen} and \eqref{eq:discont-gen} the above limit exists and is continuous at $u=0$. If $t$ is a continuity point of $A$, we interpret the integral term in the last line as the log-characteristic function of a compound Poisson distribution with intensity $1/\brac{A_t-A_{t-h_n}}$, which is infinitely divisible. This implies that also their weak limit is infinitely divisible. We conclude that the r.h.s. of \eqref{eq:cont-gen} is the log-characteristic functions of an infinitely divisible distributions and therefore of L\'evy-Khintchine form. From here the admissibility of $\brac{\alpha,\beta,\mu}$ at points of continuity of $A$ follows on the same lines as in \cite{DuffieFilipovicSchachermayer}. 
For discontinuity points $t \in J^A$ of $A$, we obtain from \eqref{eq:discont-gen} that
\[\exp\left(\gamma_0(t,u) + \scal{\bar \gamma(t,u) + u}{x}\right) = \int_D f_u(\xi) p_{t-,t}(x,d\xi),\]
where we have written $p_{t-,t}(x,.)$ for the weak limit of $p_{t-h,t}(x,.)$ as $h \downarrow 0$. Part (ii) of the admissibility conditions follows form the fact that $p_{t-,t}(x,.)$ must be supported on $D$ for all $x \in D$ and $0 \le s \le t$. If $X$ is infinitely divisible, then the decomposition of $\gamma$ as \eqref{eq:gamma_decomp1}, together with standard support theorems for infinitely divisibly distributions (cf. \cite[Ch.~24]{Sato}) yield (ii').
\end{proof}


In the remaining part of the section we show the following: Given an admissible enhanced parameter set, we can construct a Markov process that is an infinitely divisible affine semimartingale for every starting point in $D$. In this regard we require a further integrability assumption.

\begin{assumption}\label{a:integrable}
Given an enhanced parameter set $(A,\beta,\alpha, \mu)$, assume that $\alpha,\;\beta$ and $\mathcal{M}$ defined by \eqref{eq:calM} are locally integrable with respect to $A$.
\end{assumption}

\begin{proposition}\label{prop:ex-riccati}
Let $(A,\alpha,\beta,\mu)$ be an admissible enhanced parameter set satisfying Assumption \ref{a:integrable}. Then, for all $(T,u) \in (0,\infty) \times \mathcal{U}^\circ$ there exists a unique solution $\brac{\phi_.(T,u),\psi_.(T,u)}$ on $[0,T]$ to the generalized measure Riccati equations \eqref{char:affine}-\eqref{eq:initial-cond} (or equivalently to \eqref{eq:unified-riccati}).
\end{proposition}

In the following let $u=(v,w)\in\mathcal{U}$ with $v\in\C^m_{\leq0}$ and $w\in i\R^n$. We will also use the convention $\int_{(a,b]}=\int_a^b$ to shorten notation in some places.

\begin{proof} Since an enhanced parameter set is given, the generalized measure Riccati equations \eqref{char:affine}-\eqref{eq:initial-cond} can be combined into \eqref{eq:unified-riccati}. It suffices to show existence of a unique global solution to equation \eqref{eq:unified-riccati-psi}, since existence and uniqueness for \eqref{eq:unified-riccati-phi} then follows by simple integration (note that $\phi$ does not appear on the right hand side of \eqref{eq:unified-riccati-phi}). Due to the admissibility conditions the equation for $\psi$ can be split into an equation for the components $\psi^{\mathcal{I}}= (\psi^i) ,i\in\mathcal{I}$ and a decoupled linear equation for the components with $j\in\mathcal{J}$ (see also \cite[Sec.~6]{DuffieFilipovicSchachermayer}), which can be written as:
\begin{equation*}
\psi^{\mathcal{J}}_t\brac{T,u}=w+\int_t^T \bar{\beta}_{JJ}\brac{s} \psi^{\mathcal{J}}_s\brac{T,u} dA_s.
\end{equation*}
This linear equation can be solved according to Example \ref{ex:linear-equation} in the Appendix which yields a function with linear dependency on the starting value $w$, i.e., 
\begin{equation}
\psi^{\mathcal{J}}_t\brac{T,u}=w\psi^\mathcal{J}_t\brac{T},\quad \,\psi^\mathcal{J}\brac{T}: [0,T] \to \R^{n\times n}.
\label{eq:psi-J-linear}
\end{equation}

The existence and uniqueness of a \emph{local} solution to the generalized measure Riccati equation \eqref{eq:unified-riccati-psi} is a consequence of Theorem \ref{thm:ex-unique-de} in the appendix. Indeed, $\mathfrak{R}(t,(v,w))$ is of L\'evy-Khintchine form, hence analytical in $v$ by Lemma 5.3(i) in \cite{DuffieFilipovicSchachermayer} and thus locally Lipschitz continuous in $u$ with a Lipschitz constant that can be chosen $A$-integrable, due to the integrability of the enhanced parameters $\brac{\alpha, \beta, \mu}$.
To extend the local solution to the entire time-horizon we adopt the proof in \cite{DuffieFilipovicSchachermayer} to our setting. Let $g\brac{\cdot,T,u}$ be a local solution to the Riccati equations with terminal condition $u\in\mathcal{U}^\circ$ at time $T$. We have to show that $g$ extends -- backwards in time -- to a global solution on $[0,T]$. Consider the life-time of $g$ in $\mathcal{U}^{\circ}$
\begin{equation*}
\label{eq:tau-T-u}
\tau_{T,u}\coloneqq \limsup_{n\rightarrow\infty}\set{t\in \R_+\vert \Norm{g\brac{t,T,u}}\geq n \mbox{ or }g\brac{t,T,u}\in \brac{\mathcal{U}^\circ}^\complement}.
\end{equation*}
For the existence on the entire time horizon $\tau_{T,u}$ has to be zero, for all $u\in\mathcal{U}^\circ$.
Similar to \cite[equation (6.8)]{DuffieFilipovicSchachermayer} we obtain from the L\'evy--Khintchine form of $\mathfrak{R}$ for $dA$-almost-all $t$ that 
\begin{equation}
\label{eq:estimation-R-cont}
\mbox{Re } \mathfrak{R}_i\brac{t,u}\leq C(t)  \brac{\brac{\mbox{Re }u_i}^2 - \mbox{Re }u_i},
\end{equation}
where $C(t)$ is a constant independent of $u$, for all $t$. The integrability of the parameters of $\mathfrak{R}$ allows to choose $C$ as also being $A$-integrable. Hence the local solution $g$ satisfies the following integral inequality 
\begin{eqnarray*}
\mbox{Re }g^i_t\brac{T,u}&\leq&v+\int_{(t,T]}C\brac{s}\brac{\brac{\mbox{Re } g^i_s\brac{T,u}}^2-\mbox{Re }g^i_s\brac{T,u}}dA_s.
\end{eqnarray*}
By the comparison result Proposition \ref{prop:comparison} for measure differential equations, stated in the appendix,  we get
\begin{equation*}
\mbox{Re }g^i_t\brac{T,u}\leq f_t\brac{T,u}
\end{equation*}
where $f$ satisfies
\begin{equation*}
f_t\brac{T,u}=\Re v+\int_{(t,T]}C\brac{s}\brac{f_s\brac{T,u}^2-f_s\brac{T,u}}dA_s.
\end{equation*}
Note that for all $K>0$ there exists $c>0$ such that $(x^2-x)<-cx$ as long as $x\in \brac{-K,0}$. Hence, $f_\cdot\brac{T,u}<0$ for all $u\in\mathcal{U}^\circ$.

For the upper bound we consider the squared norm of $\psi^\mathcal{I}$. With the chain rule formula for functions of bounded variation in \cite[Theorem 4.1]{CrastaCicco2011} we can write
\begin{equation}
\begin{split}
\Norm{\psi_t^\mathcal{I}\brac{T,u}}^2=&\Norm{v}^2+\int_{(t,T]} 2 \Re\scal{\overline{\psi_s^\mathcal{I}\brac{T,u}}}{\mathfrak{R}^\mathcal{I}\brac{s,\psi_s^\mathcal{I}\brac{T,u},\psi_s^\mathcal{J}\brac{T,u}}}dA^c_s.
\\
& + \sum_{s\in(t,T]}\Norm{\psi_s^\mathcal{I}\brac{T,u}}^2-\Norm{\psi_{s-}^\mathcal{I}\brac{T,u}}^2
\\
=&\Norm{v}^2+\int_{(t,T]} 2 \Re\scal{\overline{\psi_s^\mathcal{I}\brac{T,u}}}{\mathfrak{R}^\mathcal{I}\brac{s,\psi_s^\mathcal{I}\brac{T,u},\psi_s^\mathcal{J}\brac{T,u}}}dA_s.
\\
& - \sum_{s\in(t,T]} 
\scal{\overline{\Delta\psi_s^\mathcal{I}\brac{T,u}}}{\Delta \psi_s^\mathcal{I}\brac{T,u}}
\\
\leq&\Norm{v}^2+\int_{(t,T]} 2 \Re\scal{\overline{\psi_s^\mathcal{I}\brac{T,u}}}{\mathfrak{R}^\mathcal{I}\brac{s,\psi_s^\mathcal{I}\brac{T,u},\psi_s^\mathcal{J}\brac{T,u}}}dA_s.
\end{split}
\label{eq:norm-psi-squared}
\end{equation}
where we have used $\psi_{s-}^\mathcal{I}\brac{T,u} = \psi_{s}^\mathcal{I}\brac{T,u}-\Delta \psi_{s}^\mathcal{I}\brac{T,u}$ in the second line. With
\begin{equation*}
K\brac{t,u}\coloneqq \Re v_i\scal{\alpha^i_{\mathcal{JJ}}\brac{t}w}{w}+\Re \bar{v}_i\scal{\beta_i\brac{t} - H_i(t)}{u},
\end{equation*}
we can write
\begin{eqnarray*}
\Re\brac{\bar{v}_i \mathfrak{R}_i\brac{t,u}}&=&\alpha^i_{ii}\brac{t}\norm{v_i}^2\Re v_i +K\brac{t,u}
\\
&&+\Re\brac{\bar{v_i}\int_{D\setminus\set{0}}\brac{e^{\scal{u}{\xi}}-1-\scal{u_{\mathcal{J}\cup i}}{h_{\mathcal{J}\cup i}\brac{\xi}}}\mu_i\brac{t,d\xi}}.
\end{eqnarray*}
Using the same calculations as Proposition 6.1 in \cite{DuffieFilipovicSchachermayer} we obtain the following estimate:
\begin{equation*}
\Re \brac{\bar{v}_i \mathfrak{R}_i\brac{t,u}}\leq C_t\brac{1+\Norm{w}^2}\brac{1+\Norm{v}^2},\quad \forall u=(v,w)\in\mathcal{U}.
\end{equation*}
From the $A$-integrability of $\mathcal{M}$ it follows that $C$, which is independent of $u$, can be chosen $A$-integrable. Inserting the above equation into \eqref{eq:norm-psi-squared} we obtain
\begin{equation*}
\Norm{\psi^{\mathcal{I}}_t\brac{T,u}}^2\leq \Norm{v}^2+\int_{(t,T]} C_s\brac{1+\Norm{\psi^{\mathcal{J}}_s\brac{T,u}}^2}\brac{1+\Norm{\psi^{\mathcal{I}}_s\brac{T,u}}^2}dA_s.
\end{equation*}
Gronwalls inequality for measure differential equations (c.f. \cite[Corollary 19.3.3]{Gil2007}) yields
\begin{eqnarray}
\Norm{\psi^{\mathcal{I}}_t\brac{T,u}}^2\leq \Norm{v}^2\exp\brac{\int_{(t,T]} C_s\brac{1+\Norm{\psi^{\mathcal{J}}_s\brac{T,u}}^2}dA_s}.
\label{eq:uniform-boundedness-psi}
\end{eqnarray}
With \eqref{eq:psi-J-linear} this shows that the solution can not explode 
and thus $\tau_{T,u}=0$, i.e., we have a solution on $[0,T]$.
\end{proof}

\begin{proposition}
\label{prop:properties-phi-psi}Let $\left(\phi,\psi\right)$ be a
solution to the generalized measure Riccati equations  (\ref{ric1})-(\ref{eq:initial-cond}). Then it holds that
\begin{enumerate}[(i)]
\item for each $u\in \cU$ and $s<t$ the left limits
\[
\phi_s\left(t\uminus,u\right)=\lim_{\varepsilon\downarrow0}\phi_s\left(t-\varepsilon\right)\mbox{, and }\psi_s\left(t\uminus,u\right)=\lim_{\varepsilon\downarrow0}\psi_s\left(t-\varepsilon,u\right)
\]
exist.

\item For all $u=\brac{v,w}\in\mathcal{U}$ and $s\leq t$, $\psi_s^{\mathcal{J}}\left(t,\left(v,0\right)\right)=0$.

\item\label{item:semiflow} $\brac{\phi,\psi}$ satisfy the semiflow property, i.e. let $r\leq s\leq t$ then for all $u\in\mathcal{U}^\circ$ 
\begin{align*}
\phi_{r}\left(t,u\right)  &=  \phi_{s}\left(t,u\right)+\phi_{r}\left(s,\psi_{s}\left(t,u\right)\right)&\quad\mbox{and }\phi_{t}\brac{t,u}=0,
\\
\psi_{r}\left(t,u\right)  &=  \psi_{r}\left(s,\psi_{s}\left(t,u\right)\right)&\quad\mbox{and }\psi_{t}\left(t,u\right)=u.
\end{align*}

\item\label{item:uniform-boundedness} For all $t\in\brak{0,T}$ and $K\subset\mathcal{U}$ compact $$\sup_{u\in K, s\leq t}\Norm{\psi_s\brac{t,u}}<\infty.$$

\end{enumerate}
\end{proposition}

\begin{proof}
The first assertion follows from the integral representation of $\phi$ and $\psi$. The second assertion can be derived directly from the admissibility conditions. Regarding \eqref{item:semiflow}, let $s\leq t,\; u\in\mathcal{U}^\circ$ and define
\begin{equation*}
f\brac{r}\coloneqq \psi_r\brac{s,\psi_s\brac{t,u}}, \mbox{ for } 0\leq r\leq s.
\end{equation*}
Plugging equation \eqref{eq:unified-riccati-psi} into the above definition we see that - on $[0,s]$ - $f$ satisfies the same measure Riccati equation  as $\psi_r\brac{t,u}$:
\begin{eqnarray*}
f(r) &=& \psi_s\brac{t,u}+\int_{(r,s]} \mathfrak{R}\BRac{w,f(w)}dA_s
\end{eqnarray*}
By  uniqueness  of the Riccati equation we infer $f\brac{r}=\psi\brac{r,t,u}$. A simple calculation exploiting the above and equation \eqref{eq:unified-riccati-psi} shows the equation for $\phi$.
Assertion \eqref{item:uniform-boundedness} follows readily from equations \eqref{eq:psi-J-linear} and \eqref{eq:uniform-boundedness-psi}.
\end{proof}

We are now prepared to state our main result on existence of affine Markov processes and affine semimartingales:

\begin{theorem}\label{thm:existence}
Let $(A,\alpha, \beta, \mu)$ be an admissible enhanced parameter set satisfying Assumption~\ref{a:integrable}. Then there exists an infinitely divisible affine Markov process $X$ (cf. Definition~\ref{def:affine_markov}) with $\phi, \psi$ solutions of the associated measure Riccati equations.
If $X$ is conservative, then it is an affine semimartingale with characteristics given by \eqref{char:affine}, for any initial point $X_0 = x \in D$.
\end{theorem}

The next result provides a sufficient condition for the conservativeness of $X$; further conditions can be developed along the lines of \cite[Lem.~9.2]{DuffieFilipovicSchachermayer}.

\begin{corollary}\label{cor:existence}
Let $X$ be an affine Markov process as in Theorem~\ref{thm:existence}. If, for any $T > 0$, 
$g\equiv 0$ is the only $\R^m_{\leq 0}$-valued solution to 
\begin{equation}
\frac{d g_t}{dA_t}= - \Re \mathfrak{R}^\mathcal{I}\brac{t,g_t},\quad g_T=0,
\label{eq:condition-conservative}
\end{equation}
then $X$ is conservative.
\end{corollary}

Theorem~\ref{thm:existence} follows almost entirely from the next two Propositions:

\begin{proposition}
\label{prop:ex-markov}
Let the assumptions of Theorem~\ref{thm:existence} hold true and let $\brac{\phi,\psi}$ be solutions to \eqref{ric1}-\eqref{eq:initial-cond} with admissible parameters. Then there exists an affine Markov process $X$, unique in law, with state space $D$ and whose transition kernels satisfy the affine property \eqref{eq:Markov} with exponents $\phi$ and $\psi$. 
\end{proposition}

For the proof of this proposition we introduce the following notation (see \cite[Sec.~7]{DuffieFilipovicSchachermayer}). Let $\mathcal{C}$ denote the convex cone of functions $\phi\colon \mathcal{U}\rightarrow\mathbb{C}_{\le 0}$ of the form
\begin{equation}
\phi\brac{u}=\scal{Aw}{w}+\scal{B}{u} -C + \int_{D\setminus\set{0}} \brac{e^{\scal{u}{\xi}}-1-\scal{w}{h_{\mathcal{J}}\brac{\xi}}}M\brac{d\xi}
\label{eq:convex-cone-definition}
\end{equation}
for $u\coloneqq\brac{v,w}\in\mathcal{U}$, where $A\in \cS_+^d$, $B\in D$, $C \in \Rplus$ and $M(d\xi)$ is a nonnegative Borel measure on $D\setminus\set{0}$ integrating $\scal{\mathbf{1}}{h_{\mathcal{I}}\brac{\xi}}+\Norm{h_{\mathcal{J}}\brac{\xi}}^2$. We denote by $\mathcal{C}^m$ the $m$-fold cartesian product of $\mathcal{C}$. Recall from \cite[Lemma~7.1]{DuffieFilipovicSchachermayer}, that $\phi \in \mathcal{C}$ if and only if there exists a sub-stochastic measure $\eta$ on $D$ such that
\begin{equation}\label{eq:substoch}
\int_D {e^{\scal{\xi}{u}}} \eta(d\xi) = e^{\phi(u)}, \qquad \forall\,u \in \cU.
\end{equation}

\begin{proof}
The proof splits into four steps. First, we show, under some restrictions on the form of $\mathfrak{F},$ and $\mathfrak{R}$, that the solutions $(\phi,\psi)$ of the generalized measure Riccati equations are in $\mathcal{C} \times  \mathcal{C}^d$, which follows similar to Proposition 7.4 (ii) in \cite{DuffieFilipovicSchachermayer}. In concrete terms, suppose that, for all $i\in \mathcal{I}$,
\begin{gather}
\int_{D\setminus \set{0}}h_i\brac{\xi} \mu_i\brac{d\xi}<\infty
\nonumber
\\
\alpha_{i,ik}=\alpha_{i,ki}=0,\mbox{ for all }k\in \mathcal{J}
\label{eq:condition1-curly-c}
\end{gather}
In this case $\mathfrak{R}^{\mathcal{I}}$ can be written in the form 
\begin{equation*}
\mathfrak{R}_i^{\mathcal{I}}(t,u)=\tilde{\mathfrak{R}}_i^{\mathcal{I}}(t,u)-c_i(t) v_i,\quad i\in\mathcal{I}
\end{equation*}
with $\tilde{\mathfrak{R}}_i\in\mathcal{C}$, $c_i \ge 0$ $dA$-a.e. and $c_i(t) \Delta_t A \le 1$. Therefore, the generalized measure Riccati equation \eqref{eq:unified-riccati-psi} is equivalent to the following equation:
\[\psi^i_{s}\brac{t,u} = v_i \,\cE_s^t(-c_i dA) + \int_s^t \cE_r^s(-c_i dA) \tilde{\mathfrak{R}}\brac{r,\psi_r\brac{t,u}}dA_r,\quad i\in\mathcal{I},\]
where 
\[\cE_s^t(-c_i dA) = \exp\left(-\int_s^ t c_i(r)dA^c_r\right) \prod_{r \in (s,t]} (1 - c_i(r) \Delta A_r)\]
is the solution to the linear measure differential equation $\tfrac{dg_t}{dA_t} = c_i(t) g_t$, see Example~\ref{ex:linear-equation}.
Define the iterative sequence 
\begin{eqnarray*}
{}^{\brac{0}}\psi^{i}_s\brac{t,u} &=& v_i,
\\
{}^{\brac{k+1}}\psi^{i}_s\brac{t,u} &=& v_i \,\cE_s^t (-c_i dA) + \int_s^t \cE_r^s(-c_i dA) \tilde{\mathfrak{R}_i}\brac{r,{}^{\brac{k}}\psi_r^{\mathcal{I}}\brac{t,u},\psi_s^{\mathcal{J}}\brac{t,u}}dA_r.
\end{eqnarray*}
By Banachs fixed point theorem and Helly's selection principle 
there is a subsequence of $\brac{{}^k\psi^{\mathcal {I}}}_{k\in\mathbb{N}}$ that converges pointwise to the solution $\psi^{\mathcal{I}}$ of \eqref{eq:unified-riccati-psi}. By Proposition 7.2 in \cite{DuffieFilipovicSchachermayer} $\mathcal{C}^m$ is stable under composition and pointwise limits and we conclude that $\psi^{\mathcal{I}}_s\brac{t,\cdot}\in\mathcal{C}^m$. The assertion $\psi^{\mathcal{J}}_s\brac{t,\cdot}\in\mathcal{C}^n$ follows directly from \eqref{eq:psi-J-linear}. Since $\mathfrak{F}$ is in $\mathcal{C}$ 
also $\phi_s\brac{t,\cdot}$ is in $\mathcal{C}$, cf. \cite[Prop.~7.2]{DuffieFilipovicSchachermayer}.

Second, we prepare for the approximation argument of part three and establish continuous dependence of a solution to the generalized measure Riccati equations on the right hand side, i.e. convergence in $L^1\brac{dA}\times(\mbox{uoc. on }\mathcal{U})$ of the right hand side implies convergence of the solution in $(dA-\mbox{a.e.})\times(\mbox{uoc. on }\mathcal{U})$.
Here and in the following, `uoc. on $\cU$' means uniformly on compact subsets of $\cU$. Indeed, let $K\subseteq \mathcal{U}$ compact and $\mathfrak{R}, \tilde{\mathfrak{R}}$ with good, admissible and $A$-integrable parameters, such that 
\begin{equation}
\Norm{\sup_{u\in K} \left(\mathfrak{R}\brac{\cdot,u} - \tilde{\mathfrak{R}}\brac{\cdot,u}\right)}_{L^1\brac{dA}}\leq \delta.
\end{equation}
Denote the solution corresponding to $\tilde{\mathfrak{R}}$ by $\tilde{\psi}$ and examine the difference with $\psi$:
\begin{eqnarray*}
\norm{\psi_t\brac{T,u}-\tilde{\psi}_t\brac{T,u}} &\leq& \int_t^T\norm{\mathfrak{R}\brac{\vphantom{\tilde{\psi}}s,\psi_s\brac{T,u}}-\tilde{\mathfrak{R}}\brac{s,\tilde{\psi}_s\brac{T,u}}}dA_s
\\
&\leq&\int_t^T \norm{\mathfrak{R}\brac{\vphantom{\tilde{\psi}}s,\psi_s\brac{T,u}}-\mathfrak{R}\brac{s,\tilde{\psi}_s\brac{T,u}}}dA_s
\\
&&+\int_t^T \norm{{\mathfrak{R}}\brac{s,\tilde{\psi}_s\brac{T,u}}-\tilde{\mathfrak{R}}\brac{s,\tilde{\psi}_s\brac{T,u}}}dA_s.
\end{eqnarray*}
If $\tilde{\psi}$ stays in $K$ we can estimate the second summand by $\delta$ and obtain with Proposition~\ref{prop:properties-phi-psi}(iv) in conjunction with the local Lipschitz-continuity of $\mathfrak{R}$ (with $A$-integrable Lipschitz constant) that 
\begin{equation}
\norm{\psi_t\brac{T,u}-\tilde{\psi}_t\brac{T,u}}\leq \delta+\int_t^T L_s\norm{\psi_s\brac{T,u}-\tilde{\psi}_s\brac{T,u}}dA_s.
\end{equation}
By Gronwalls lemma for Stieltjes differential equation (c.f. \cite[Corollary 19.3.3]{Gil2007}) the difference satisfies
\begin{equation}
\norm{\psi_t\brac{T,u}-\tilde{\psi}_t\brac{T,u}}\leq \delta \exp\brac{\int_t^T L_s dA_s}.
\label{eq:Gronwall-Result-ineq}
\end{equation}
Now suppose
\begin{equation}
\tau=\sup\left\{t\in\brak{0,T}\colon \norm{\psi_t\brac{T,u}-\tilde{\psi}_t\brac{T,u}}>\alpha\right\}>0
\end{equation}
This implies that the difference of $\psi$ and $\tilde \psi$ is less than $\alpha$ for all $t\in\Brak{\tau,T}$ due to the common terminal value of $\psi$ and $\tilde{\psi}$ and the continuity from the right. By \eqref{eq:Gronwall-Result-ineq} we can choose $\delta$ small enough, such that $\norm{\psi_t\brac{T,u}-\tilde{\psi}_t\brac{T,u}}\leq \frac{\alpha}{L_\tau\Delta A_\tau+1} \le \alpha $. Therefore $\tilde{\psi}$ can not leave the $\alpha$-neighborhood continuously, but only by a jump. However, $\psi$ satisfies 
$$\Delta \psi_t\brac{T,u}=\mathfrak{R}\brac{t,\psi_t\brac{T,u}}\Delta A_t $$
at points of discontinuity (similarly for $\tilde \psi$) from which it follows that $\norm{\psi_{\tau-}\brac{T,u}-\tilde{\psi}_{\tau-}\brac{T,u}} < \alpha$ - a contradiction. This proves the continuous dependence on the right hand side.

Third, we show an analogue of \cite[Lemma 5.7]{Filipovic05}, i.e.~that there exists a sequence $\brac{\mathfrak{R}_k}_{k\in\mathbb{N}}$ of functions of L\'evy Khintchine form with admissible parameters satisfying Assumption~\ref{a:integrable} and conditions \eqref{eq:condition1-curly-c}, converging to $\mathfrak{R}$ in $(L^1\brac{dA})\times(\mbox{uoc.on }\mathcal{U})$.

The construction of the sequence $\brac{\mathfrak{R}_k}_{k\in\mathbb{N}}$ of functions satisfying \eqref{eq:condition1-curly-c} is the same as in the proof of \cite[Lemma 5.7]{Filipovic05} or \cite{DuffieFilipovicSchachermayer} p. 33. Only the mode of convergence has been strengthened to convergence in $L^1(dA)\times (\mbox{uoc.on }\mathcal{U})$. From \cite{DuffieFilipovicSchachermayer} p. 33 we obtain, for any $i \in \mathcal{I}$, the identity

\begin{equation} \tilde{\mathfrak{R}}^i_k\brac{t,u}-\mathfrak{R}^i\brac{t,u}=\frac{2}{p_i^\ast\brac{t}}\brac{h_u\brac{\frac{\xi^\ast\brac{t}}{k}} - \frac{1}{2}\scal{Q(t)u_{\mathcal{J} \cup i}}{u_{\mathcal{J} \cup i}}},
\label{eq:difference-R-R} 
\end{equation}
where
$$ p^\ast\brac{t} = \frac{\alpha^i_{ii}\brac{t}}{\Norm{\alpha^i_{i\mathcal{J} \cup i}\brac{t}}^2},\quad \xi^\ast(t)_{\mathcal{I}\setminus i}=0,\;\xi(t)_{\mathcal{J} \cup i}^\ast = \frac{\alpha^i_{i\mathcal{J} \cup i}\brac{t}}{\Norm{\alpha^i_{i\mathcal{J} \cup i}\brac{t}}},  \quad Q(t)_{kl}\coloneqq p^\ast \frac{\alpha^i_{ki}\alpha^i_{il}}{\alpha^i_{ii}}$$
$$h_u\brac{\xi}=\brac{e^{\scal{u}{\xi}}-1-\scal{u_{\mathcal{J} \cup i}}{h_{\mathcal{J} \cup i}(\xi)}} / \left(\scal{\mathbf{1}}{h_{\mathcal{I} \setminus i}\brac{\xi}}+\Norm{h_{\mathcal{J} \cup i}\brac{\xi}}\right)$$
We can simplify the expressions in \eqref{eq:difference-R-R} to
\begin{equation*}
\frac{1}{2}\scal{Q(t)u_{\mathcal{J} \cup i}}{u_{\mathcal{J} \cup i}} = 2\sum_{l,m\in \mathcal{J} \cup i} u_l \frac{\alpha^i_{li}(t)\alpha^i_{im}(t)}{\alpha^i_{ii}(t)}u_m.
\end{equation*}
Using the properties of the truncation functions and $\Norm{\xi^\ast}=1$, we obtain for large enough $k$ that
\begin{equation*}
\frac{2}{p^\ast\brac{t}}h_u\brac{\frac{\xi^\ast\brac{t}}{k}}\leq C\frac{1}{p^\ast\brac{t}}\brac{1+\Norm{u_{\mathcal{J} \cup i}}^2\Norm{\xi^\ast}^2}\leq C\brac{1+\Norm{u_{\mathcal{J} \cup i}}^2}\frac{\Norm{\alpha^i_{i\mathcal{J} \cup i}\brac{t}}^2}{\alpha^i_{ii}\brac{t}}
\end{equation*}
where $C$ does not depend on $u$ or $\xi^\ast$. Integrability of the above quantities w.r.t. $A$ follows from the positive semi-definiteness of $\alpha(t)$ and the Cauchy-Schwartz inequality. This implies convergence of $\tilde{\mathfrak{R}}_k$ to $\mathfrak{R}$ in $(L^1(dA))\times\brac{\mbox{uoc.on } \mathcal{U}}$ due to the construction of $\tilde {\mathfrak{R}}$.

Finally, we come to the last step.  From \eqref{eq:substoch} it now follows, that for every $\brac{t,x}\in\brak{0,T}\times D$ and $s\in\brak{0,t}$, there exists a unique, sub-stochastic measure $p_{s,t}\brac{s,\cdot}$ on $D$ with 
\begin{equation}
\int_D e^{\scal{u}{\xi}}p_{s,t}\brac{x,d\xi}=e^{\phi\brac{s,t,u}+\scal{\psi\brac{s,t,u}}{x}},\quad \forall u\in\mathcal{U}.
\label{eq:stochastic-measure}
\end{equation}
The semiflow property of $\brac{\phi,\psi}$ ensures that the family of measures $\brac{p_{s,t}}_{s\leq t\in\brak{0,T}}$ satisfies the Chapman-Kolmogorov equations. By the Kolmogorov existence theorem (see~\cite[Theorem 8.4]{Kallenberg2002}), there exists a $D$-valued Markov process $X$ on $\brak{0,T}$, unique in law, with transition kernels $\brac{p_{s,t}}_{s\leq t\in\brak{0,T}}$. By definition, $X$ satisfies the affine property \eqref{eq:Markov} for all $u\in\mathcal{U}$.
\end{proof}

\begin{proposition}
\label{prop:semimart}Let $X$ be the affine Markov process from Proposition~\ref{prop:ex-markov} started at some $X_0 = x \in D$. If $X$ is conservative, then there is a modification of X which is a c\`adl\`ag affine semimartingale.
\end{proposition}
\begin{proof}
Let $X$ be the affine Markov process and $(\cF_t)_{t \geq 0}$ its natural filtration. From \eqref{eq:Markov}, we have that
\begin{equation}
M_t^{T,u} := \E\brak{\left. e^{\scal{u}{{X}_T}}\right\vert \cF_t}=e^{\phi_t\brac{T,u}+\scal{\psi_t\brac{T,u}}{{X}_t}},
\end{equation}
which must be a martingale for all $u \in \cU$. Since $\phi$ and $\psi$ are right-continuous in $T$ and c\`adl\`ag in $t$, applying this identity with $t=0$ shows that  $X$ (and therefore also every $M^{T,u}$) is right-continuous in probability. It follows that the martingale $M^{T,u}$ has a c\`adl\`ag modification. Let $u=\brac{v,w}$. By equation \eqref{eq:psi-J-linear} $\psi_t^{\mathcal{J}}\brac{T,\brac{v,0}}=0$ for all $t<T$ and hence $\langle\psi^{\mathcal{I}}_t\brac{T,(v,0)}, X_t^{\mathcal{I}} \rangle$ are c\`adl\`ag semimartingales for $v\in\R^m_-$ on $[0,T]$. For some linearly independent vectors $e_1,\dots,e_m$ in $R^m_{\le 0}$ we can find $s\le T$ such that $\psi^{\mathcal{I}}_t\brac{T,e_1},\dots,\psi^{\mathcal{I}}_t\brac{T,e_m}$ are linearly independent for all $t \in (s,T]$. Thus $X^{\mathcal{I}}$ is a semimartingale on $(s,T]$. This can be done for arbitrary $T$ which allows to infer with a covering argument (and right-continuity at $t=0$), that $X^{\mathcal{I}}$ is a semimartingale on $\Rplus$.

For the real valued part $X^\mathcal{J}$ of the process we use that, for all $u=\brac{v,w}\in\mathcal{U}^{\circ}$, the equation for $\psi^\mathcal{J}$ reduces to a linear equation with solution $\psi_t^\mathcal{J}\brac{T,u}=w\psi_t^\mathcal{J}(T)$ (see equation \eqref{eq:psi-J-linear}). By the same argument as in \cite[Proof of Theorem 2.12]{DuffieFilipovicSchachermayer} it follows that also $X^\mathcal{J}$, is a c\`adl\`ag semimartingale.
\end{proof}

We complete the proof of Theorem~\ref{thm:existence} and Corollary \ref{cor:existence}.
\begin{proof}
In light of Propositions~\ref{prop:ex-markov} and \ref{prop:semimart} it only remains to show that the semimartingale triplet of $X$ is given by \eqref{char:affine} with the same parameters that were used for the construction of  $X$. To this end, we apply Lemma \ref{lem:characteristics} to $X$, and get, similar to equation \eqref{eq:phi_to_G},
\begin{equation*}
\Theta_{t}(\omega) \cdot 
\begin{pmatrix} F\brac{t,\psi_t\brac{T,u}}\\ R^1\brac{t,\psi_t\brac{T,u}} \\ \vdots \\ R^d\brac{t,\psi_t\brac{T,u}}
\end{pmatrix}dA^c_t
=\begin{pmatrix} G_{0}(dt;\omega,T,u) \\ \vdots \\ G_{d}(dt;\omega,T,u),
\end{pmatrix}
\end{equation*}
where $F, R$ on the left hand side contain the parameters $(A,\beta, \alpha, \mu)$ and $G$ the semimartingale characteristics of $X$ (cf. \eqref{eq:G}). We proceed as in the proof of Theorem \ref{prop1} by taking the union over a countable, dense subset $\cT\times\cE$ of $\R_{\geq 0}\times\cU$ and considering the right limits $T\downarrow t$ in the countable set $\cT$. Using $\psi_t\brac{t,u}=u$ and the fact that functions of L\'evy--Khintchine-form determine their parameter triplets uniquely, we derive the continuous part of \eqref{char:affine} . The equation for $\nu$ at jump points follows from Lemmata~ \ref{lem:affinejumps} and \ref{lem:id_jumps}, completing the proof of Theorem~\ref{thm:existence}.

For the proof of Corollary~\ref{cor:existence}, evaluating \eqref{eq:Markov} at $u = 0$ yields
\begin{equation}\label{eq:kernel_conservative}
p_{t,T}(x,D) = \exp\left(\phi_t(T,0) + \scal{\psi_t(T,0)}{x}\right)
\end{equation}
for all $0 \le t \le T$ and $x \in D$. Taking into account that $p_{t,T}(x,D) \le 1$ and that $D = \Rplus^m \times \RR^n$, we see that $\phi_t(T,0) \le 0$, $\psi^\cI_t(T,0) \le 0$ and $\psi^\cJ_t(T,0) = 0$. Writing $g(t) := \psi^\cI_t(T,0)$ the measure Riccati equation \eqref{eq:unified-riccati-psi} becomes \eqref{eq:condition-conservative}. This equation has the constant solution $g \equiv 0$; if it is the only solution, then $\psi^\cI_t(T,0) = 0$ for all $0 \le t \le T$. Inserting into \eqref{eq:unified-riccati-phi}, also $\phi_t(T,0) = 0$. Together with \eqref{eq:kernel_conservative}, this shows that $p_{t,T}(x,D) = 1$, i.e. that $X$ is conservative.
\end{proof}

\begin{remark}The proof of Theorem~\ref{thm:existence} can easily be adapted to the case where $\gamma_0$ is not of the L\'evy--Khintchine form \eqref{eq:gamma_decomp1} at $t \in J^A$, but a general log-characteristic function of a $D$-valued random variable. This is due to the fact that $\gamma_0$ enters only into part \eqref{eq:unified-riccati-phi}, but not into part \eqref{eq:unified-riccati-psi} of the measure Riccati equation. 
\end{remark}

\section{Examples and applications}\label{sec:examples}

We begin this section with some examples which illustrate several aspects of stochastic discontinuities within affine semimartingales. 

After that, we study affine semimartingales in discrete time in Section \ref{sec:discretetime}. In Section \ref{ex:dividends} we glance at the application of affine semimartingales to stock prices with dividends and in Section \ref{sec:affine} we consider a new class of affine term structure models allowing for stochastic discontinuities.

\begin{example}\label{ex1}
Consider the following discrete-time variant of the (time-inhomogeneous) Poisson process: let $X_0=x \in \N$. Furthermore, assume that $X$ is constant except for $t\in\{1,2,\dots\}$ and assume that $\Delta X_n \in \{0,1\},\ n \in \{1,2,\dots\}$ are independent  with $P(\Delta X_n=1)=p_n \in (0,1)$. Then $X$ is an affine semimartingale because for $0 \le s \le t$,
$$ E[e^{u X_t}| \cF_s] = \exp\Big( u X_s + \sum_{s<n \le t, n \in \N} \phi_n(u) \Big) $$
where
$$ \phi_n(u)  = E[e^{u\Delta X_n}] = e^u ( p_n + e^{-u}(1-p_n)) = \exp( u + \log(p_n + e^{-u}(1-p_n))). $$
Clearly, it may happen that $\Delta X_n=0$ while $ \phi(u,n,t) - \phi(u,n\uminus,t)= \phi_n(u)\neq 0$. Stochastic discontinuity is reflected by having jumps at $t \in \{1,2,\dots\}$  with positive probability. The considered process falls in the class of point processes whose associated jump measure is an \emph{extended} Poisson measure, see II.1c in \cite{JacodShiryaev}. In contrast to Poisson processes, $X$ is not quasi-left continuous.
In summary, $X$ is a process with independent increments, but not a time-inhomogeneous L\'evy process.
 \hfill $\diamond$
\end{example}

The following example illustrates how one can construct stochastically discontinuous affine semimartingales from stochastically continuous ones, even from affine semimartingales without jumps, through a suitable (discontinuous) time-change.

\begin{example}
This example is inspired by \cite{GehmlichSchmidt2015MF}:
consider an affine semimartingale $X$ which is stochastically continuous (as treated in \cite{DuffieFilipovicSchachermayer} and \cite{Filipovic05}). We assume that $D$ denotes the state space of the affine semimartingale and that $\phi$ and $\psi$ are the characteristics of $X$ as in \eqref{cond:affine}.

Let $\set{t_1<\dots<t_N}\subset\R_{\ge 0}$ be some time points and $a_i\in\R^d$, $b_i\in\R^{d\times d}$ such that $a_i+b_i\cdot x\in D$ for all $x\in D$, $i=1,\dots,N$. Then
\begin{align}\label{ex:5.2.formula}
	\tilde{X}_t:=\sum_{i=1}^N \Ind_{\set{t\geq t_i}}\brac{a_i+b_i\cdot X_{t}},\quad t \ge 0
\end{align}
is an affine semimartingale in the sense of Definition \ref{def:affine}. Note that $\tilde X$ is in general not stochastically continuous, as it jumps with positive probability at the time points $t_i$, $i=1,\dots,N$. 

Indeed, by the affine property of $X$ and using iterated conditional expectations, we obtain for $t_{k} \le t < t_{k+1}$,
\begin{align}
 	E \Big[e^{\langle u, \tilde X_t  \rangle}|\cF_{t_{k}}\Big] &= 	
 		E \Big[\exp\Big(\langle u, \sum_{i=1}^k \brac{a_i+b_i\cdot X_{t}} \rangle\Big) | \cF_{t_{k}} \Big] \notag\\  
		&= e^{\sum_{i=1}^k \langle u, a_i \rangle} 
		   E \Big[\exp\Big(\langle  \sum_{i=1}^k u b_i^\top ,  X_{t}  \rangle\Big) | \cF_{t_{k}} \Big]  \notag\\
		&= \exp\Big( \sum_{i=1}^k \langle u, a_i \rangle  + \phi_{t_k}(t,u')+ \langle \psi_{t_k}(t,u'), X_{t_k} \rangle \Big), \label{temp250}
\end{align}
since $X$ is affine; here we set $u':= \sum_{i=1}^k u b_i^\top$.
The affine characteristics of $\tilde X$ are directly obtained from Equation \eqref{temp250}. 
 \hfill $\diamond$
\end{example}

The above example suggests that even more complex variants of the transformation considered in \eqref{ex:5.2.formula} stay in the affine class. The following example shows that this need not always be the case.

\begin{example}
Consider an affine process $X$ and let 
$$ Y_t = X_t + \ind{t \ge 1} X_1, \quad t \ge 0. $$
Then $Y$ is in general not affine because for $ 1 \le s < t$,
\begin{align*}
	E[e^{u Y_t}| \cF_s] = e^{u X_1} \cdot e^{\phi_s(t,u)+\psi_s(t,u) X_s} \neq e^{\tilde \phi_s(t,u) + \tilde \psi_s(t,u) X_s}
\end{align*}
as in general $\psi_s(t,u)\neq u$. However, $(X,Y)^\top$ is affine, a property prominently used in bond option pricing.
 \hfill $\diamond$
\end{example}

The following example illustrates the possibility of processes with affine Fourier transform, which are not semimartingales: 

\begin{example} 
\label{counterexample:not_a_semimartingale}
Consider a deterministic, one-dimensional process $X_t(\omega)=f(t),$ $t \ge 0$ with a function $f$ of infinite variation. For example one may choose one path of a Brownian motion - in this case
$f$ is even continuous. Then $X$ is affine in the sense that its Fourier transform has exponential affine form, as
$$ E[e^{u X_{t}}|\cF_{s}] = e^{ u f(t)}. $$
Hence $X$ satisfies Equation \ref{def:affine} with $\phi_s(t,u)=u f(t)$ and $\psi_s(t,u)=0$. Note, however, that $X$ is not a semimartingale and that $t \mapsto \phi_s(t,u)$ is of infinite variation and hence not quasi-regular (cf.~ Definition~\ref{def:quasi_regular}). In the case of processes with independent increments the gap to those processes which are also semimartingales can be completely classified, see Section II.4.c in \cite{JacodShiryaev}. A study of the gap between affine semimartingale studied here and processes satisfying \eqref{def:affine} but which are not semimartingales is beyond the scope of this article.
 \hfill $\diamond$
\end{example}

Other than affine transitions at the discontinuity points $t_1,\dots,t_N$ are also possible, as the following example illustrates.

\begin{example}
\label{ex:poisson+jump}
Let $N$ be a Poisson process with intensity $\lambda$. This is also an affine process with affine characteristics $\psi_s\brac{t,u}=u$ and $\phi_s\brac{t,u}=\lambda\brac{t-s}\brac{e^u-1}$. Let $\alpha$ a Bernoulli distributed random Variable with $\P\brac{\alpha=-1}=\frac{1}{2}$ and $\beta$ a standard normal random variable. Further let $\alpha,\beta$ and $N$ be mutually independent. Consider a (deterministic) time $\tau>0$ and the  process given by $$X_t=N_t+\ind{t\geq \tau}\brac{\alpha+\beta \sqrt{N_\tau}}, \quad t \ge 0$$
together with the (augmented) filtration generated by $\sigma( N_s, \alpha \ind{\tau \le s}, \beta \ind{\tau \le s}: s \le t)$.  We  compute the conditional characteristic function of $X$. At first let $s<\tau\leq t$;
\begin{eqnarray}
 E\brak{e^{ \langle u, X_t \rangle}\big\vert\cF_s}
&=& E\brak{ E\brak{ e^{ \langle u, N_t+\Ind_{\set{t\geq\tau}} (\alpha+\beta \sqrt{N_{\tau}} ) \rangle }\big\vert\cF_{\tau}}\Big\vert\cF_s}
\notag\\
&=&e^{\phi_\tau\brac{t,u}} E\brak{e^{u\alpha}}\cdot E\brak{e^{\psi_\tau\brac{t,u}N_{\tau}+u\beta \sqrt{N_{\tau}}}\big\vert\cF_s}
\notag \\
&=&e^{\phi_\tau\brac{t,u}}\frac{1}{2}\brac{e^u+e^{-u}} E\brak{e^{\brac{\psi_\tau\brac{t,u}+\frac{1}{2}u^2}N_\tau}\big\vert\cF_s}
\notag\\
&=&e^{\phi_\tau\brac{t,u}}\frac{1}{2}\brac{e^u+e^{-u}}e^{\psi_s\brac{\tau,\psi_\tau\brac{t,u}+\frac{1}{2}u^2}N_s},
\notag
\end{eqnarray}
In the second case where $\tau\leq s\leq t$, we have
\begin{eqnarray}
E\brak{e^{uX_t}\vert\mathcal{F}_s}&=&\exp\brac{\phi\brac{s,t,u}+\psi\brac{s,t,u}N_s+u \big(\alpha+\beta\sqrt{N_\tau}\big)}
\notag\\
&=&\exp\Brac{\phi\brac{s,t,u}+u X_s}.
\notag
\end{eqnarray}
Hence $X$ is an affine process with affine characteristics $\tilde{\phi}$ and $\tilde{\psi}$ given by
\begin{eqnarray}
\tilde{\phi}_s\brac{t,u}&=&\phi_s\brac{t,u}+\ind{s<\tau\leq t}\Brac{\log\brac{\cosh u}}
\notag\\
\tilde{\psi}_s\brac{t,u}&=&\psi_s\brac{\tau,\psi_\tau\brac{t,u}+\ind{s<\tau\leq t}\frac{1}{2}u^2}=u+\ind{s<\tau\leq t}\frac{1}{2}u^2.
\notag
\end{eqnarray}
Note that the process $X$ does not satisfy the support condition~\ref{def:full_support}, since it is supported on the positive real whole numbers for before the jump and might take negative values after $\tau$.
\hfill $\diamond$
\end{example}

\subsection{Affine processes in discrete time}\label{sec:discretetime}
In the considered semimartingale  approach, affine processes in discrete time can also be embedded into continuous time. This allows us to  obtain a full treatment of affine processes in discrete time as special case of our general results. Note that any discrete time process is of finite variation and hence a semimartingale such that as a matter of fact, Definition \ref{def:affine} covers all discrete-time affine processes in finite dimension. 

We use the time series notation for a process in discrete time and consider without loss of generality the time points $0,1,2,\dots$
Consider a complete probability space $(\Omega,\cF,P)$ and a filtration in discrete time $\hat \bbF=(\hat \cF_n)_{n \ge 0}$.
\begin{definition}
The time series $(\hat X_n)_{n \ge 0} $ is called \emph{affine} if it is $\hat \bbF$-adapted and there exist $\C$ and $\C^d$-valued c\`adl\`ag functions $\phi_n(m,u)$ and $\psi_n(m,u)$, respectively, such that 
\begin{align}\label{cond:affine-discrete}
E\big[ e^{\scal{u}{\hat X_m}} |\hat \cF_n \big] = \exp\big( \phi_n(m,u) + \langle \psi_n(m,u), \hat X_n \rangle\big)
\end{align}
holds for all $u \in i\R^d$ and $0 \le n \le m$, $n,m \in \N_0$. It is called \emph{time-homogeneous}, if $\phi_n(m,u) = \phi_0(n-m,u)\eqqcolon\phi_{m-n}(u)$ and $\psi_n(m,u)=\psi_0(m-n,u)\eqqcolon \psi_{m-n}(u)$, again for all  $u \in i\R^d$ and $0 \le s \le t$.
\end{definition}

To emphasize the filtration we are working with, we will sometimes call $\hat X$ $\hat \bbF$-affine. 
We associate to the time series $(\hat X_n)_{n \ge0}$  the piecewise-constant embedding into continuous time
\begin{align} \label{X-Xhat} 
	X_t = \hat X_{[t]}, \qquad t \ge 0 
\end{align}
with $[t]=n$ if $n \le t < n+1$. Then $\hat X$ is c\`adl\`ag, of finite variation and hence a semimartingale. In a similar way we let $\cF_t=\hat \cF_{[t]}$ and obtain the associated filtration in continuous time. Usual conditions are not needed here.

Note that even if the affine time series is time-homogeneous, the associated continuous-time affine process $X$ will not be time-homogeneous in general: 
for $0<\epsilon<1$
\begin{align*}
E\big[ e^{\scal{u}{X_{m+\epsilon}}} |\cF_n \big] &= \exp\big( \phi_n(m+\epsilon,u) + \scal{\psi_n(m+\epsilon,u)}{X_n} \big) \\
&= \exp\big( \phi_n(m,u) + \scal{\psi_n(m,u)}{X_n} \big)
\end{align*}
which would give $\phi_{m+\epsilon-n}(u)=\phi_{m-n}(u)$ while on the other hand
\begin{align*}
E\big[ e^{\scal{u}{X_{m+\epsilon/2}}} |\cF_{n-\epsilon/2} \big] &= \exp\big( \phi_{n-\nicefrac\epsilon 2}(m+\nicefrac\epsilon 2,u) + \scal{\psi_{n-\nicefrac\epsilon 2}(m+\nicefrac\epsilon 2,u)}{X_{n-\epsilon/2}} \big) \\
&= \exp\big( \phi_{n-1}(m,u) + \scal{\psi_{n-1}(m,u)}{X_{n-1}} \big)
\end{align*}
which would give $\phi_{m-n}(u) = \phi_{m-(n-1)}(u)$ thus rendering $X$ to be constant. Time inhomogeneity in discrete time is therefore a strictly weaker concept than in continuous time. However, in the reverse direction we have a positive result.

\begin{remark} 
If $X$ is a homogeneous continuous-time $\bbF$-affine process, it follows immediately that the time-series $\hat X$ is $\hat \bbF$-affine and $\hat X$ is time-homogeneous.
\end{remark}

\begin{proposition}
Let  $(\hat X)$ be an affine time series satisyfing the support condition~\ref{def:full_support}. Then $\phi$ and $\psi$ satisfy the semiflow property
\begin{align}\label{eq:deltaphidiscrete}
\begin{aligned}
  \phi_n(m,u) & = \phi_n(n',\psi_{n'}(m,u)) + \phi_{n'}(m,u) \\
  \psi_n(m,u) & = \psi_n(n',\psi_{n'}(m,u))
\end{aligned}
\end{align}
for all $ 0 \le n <n'<\le m $,  $u \in i \R^d$.
\end{proposition}

\begin{proof}
We apply Theorem \ref{prop1}. First, note that 
\begin{align*}
z_n (u)& = \int_D e^{\scal{u}{x}} \nu(\{n\},dx)= E\Big[ \ind{\Delta X_n \neq 0}e^{\scal u {\Delta X_n}}|\cF_{n-1}\Big].
\end{align*}
Hence, 
\begin{align*}
 E\big[ e^{\scal u {\Delta X_n}}|\cF_{n-1}\big] &=  z_n(u) + P(\Delta X_n = 0 | \cF_{n-1}) = 
 z_n(u)+1-z_n(0).
\end{align*}
This yields by definition that
\begin{align}
E\big[ e^{\scal u {\Delta X_n}}|\cF_{n-1}\big] &= E\big[ e^{\scal u { X_n}}|\cF_{n-1}\big]e^{-\scal u {X_{n-1}}} = e^{\phi_{n-1}(n,u) + \scal{\psi_{n-1}(n,u)-u}{X_{n-1}}} 
\end{align}
and from Equation \eqref{eq:Deltaphipsi} we recover that  $\gamma_0(n,u)=-\phi_{n-1}(n,u)$ and $\gamma_i(n,u)=-\psi_{n-1}(n,u)+u$.  
First, theorem \ref{prop1}  yields that
$$  \Delta \phi_{n+1}(m,u) = -\phi_n(n+1,\psi_n(m,u)), $$
i.e. 
\begin{align}\label{phi1}
 \phi_n(m,u)=\phi_n(n+1,\psi_{n+1}(m,u))+ \phi_{n+1}(m,u) 
 \end{align}
for $0 \le n <m$ and all $u \in i\R^d$. By induction we obtain that  $\phi$ satisfies the semiflow propertey
\begin{align*}
    \phi_n(m,u) = \phi_n(n',\psi_{n'}(m,u)) + \phi_{n'}(m,u) 
\end{align*}
for all $0 \le n < n' < m$ and $u \in i\R^d$. In similar spirit, Theorem \ref{prop1} yields that
\begin{align*}
  \Delta \psi_{n+1}(m,u) &= - \psi_n(n+1,\psi_{n+1}(m,u)) + \psi_{n+1}(m,u)
\end{align*}
which is equivalent to
\begin{align}\label{psi1}
 \psi_n(m,u) = \psi_n(n+1,\psi_{n+1}(m,u)) 
\end{align}
and hence the semiflow property
\begin{align*}
   \psi_n(m,u) = \psi_n(n',\psi_{n'}(m,u))
\end{align*}
for all $0 \le n < n' < m$ and $u \in i\R^d$ and the claim follows. 
\end{proof}

\begin{remark}
Despite the semiflow property one obtains directly from \eqref{phi1} and \eqref{psi1} that $\phi$ and $\psi$ are unique solutions of the following difference equations 
\begin{align*}
   \phi_n(n+1) &= F(n,u) \\
   \psi_n(n+1,u) -u &= R(n,u) \\
   \phi_n(m+1,u) &= F(n,u) + \phi_n(m,u+R(m,u)) \\
   \psi_n(m+1,u) &= \psi_n(m,u+R(m,u))
\end{align*}
where the functions $F$ and $R$ are defined by the first two equations. With the notation of Theorem \ref{prop1}, $F=-\gamma_0$ and $R_i=-\gamma_i$.
These equations and the above proposition are the content of Proposition 4.4 in \cite{RichterTeichmann2017}. The authors obtain the result directly from iterated conditional expectations.
\end{remark}

\begin{example}[AR(1)]
A (time-inhomogeneous) autoregressive time series of order (1) is given by
$$ \hat X_n = \alpha(n){\hat X_{n-1}} + \epsilon_n $$
where we assume that $(\epsilon_n)$ are independent  (not necessarily identically nor normally distributed). Then, $\hat X$ is affine, as
$$ E[e^{uX_n}|\hat \cF_{n-1}]=E[e^{u\epsilon_n}]e^{\alpha(n) X_{n-1}} $$
with $\hat \cF_{n-1}=\sigma(\hat X_0,\dots,X_{n-1})$.
The generalization to higher order requires an extension of the state space. So an AR(p) series gives an affine process $(\hat X_n,\dots,\hat X_{n-p})_{n \ge p}$.
 \hfill $\diamond$
\end{example}

\subsection{Asset prices with dividends}\label{ex:dividends}

Dividends and the relationship of a firm's asset prices have been discussed and analyzed since a long time, early contributions being for example \cite{MillerModigliani1961,MillerRock1985} or the approach proposed in  \cite{Lintner1956}, for which we propose a dynamic generalization. Most notably, typical continuous-time models incorporate dividends via a dividend yield. While this approach does ease mathematical modelling it certainly does not reflect empirical facts. In this section we show how a time-inhomogeneous affine process could be used to model stock price with dividends in an efficient way.

From a general viewpoint, the following example shows how to mix two different time scales (continuous-time and discrete-time) in a time-inhomogeneous affine model. Moreover, as the discrete-time scale has a certain lag, we also show how past-dependence can be incorporated in the same way (by extension of the state space, of course).

Consider a $d\ge 3$-dimensional affine process $X$. Let $D:=X^1$ denote the cumulated dividends process where we assume that dividends are paid at the time points $t=1,2,\dots$, i.e.\ $D$ is non-decreasing and constant on each interval $[n,n+1)$, $n \ge 1$. Let $X^2$ denote the stock price process, i.e.\ the jump of $X^2$ at dividend payment dates includes subtraction of the dividend payment, $\Delta X^2_n$,  plus possibly an additional jump due to new information, for example by the height of the dividend. We will follow the approach in \cite{Lintner1956}  and assume that the size of the dividend depends linearly on the current year's profit after taxes. In this regard, let $X^3$ denote the accumulated profits of the current year after taxes, i.e.\ $X^3_n=0$ and $X^3_{n-}$ denotes the accumulated profits of the $i$th year. 

In Lintner's model, see  \cite{Lintner1956}, the current dividend $D_n$ is given by
$$ D_n = a + b X^3_{n-} + c D_{n-} + \epsilon_n, $$
where $\epsilon_n$ are mean-zero stochastic error terms. According to Theorem \ref{prop1}, $X$ may be chosen affine only if 
the conditional distribution of the $\epsilon_n$ satisfies
$$ P(\epsilon_n \in dx | X_{n-}) = \kappa_{0,3}(dx) + \sum_{i=1}^d X_{n-}^i \kappa_{i,3}(dx) $$
where for $y \in \R^d$, $\kappa_{i,j}(dx)=\int_{\R^{d-1}}\kappa(dy_1,\dots,dy_{j-1},dx,dy_{j+1},dy_d)$.
Clearly this includes for example independent error terms (not necessarily normally distributed).
The remaining components of $X$ may be used for modelling stochastic volatility or s further covariates.

\subsection{Affine term-structure models} \label{sec:affine}

In this section we study a new class of term-structure models driven by affine processes. Motivated by our findings in Section \ref{sec:characterization}, where it turned out that the semimartingale characteristics of an affine process $X$ are dominated by an increasing, c\`adl\`ag function $A$, we study the following extension of the seminal Heath-Jarrow-Morton \cite{HJM} framework: consider a family of bond prices, given by
\begin{align}\label{eq:PtT}
	P(t,T) = \exp \Big( - \int_{(t,T]} f(t,u) dA_u \Big), \qquad 0 \le t \le T \le T^*,
\end{align}
with some final time horizon $T^*>0$. The rate $f(t,T)$ is called \emph{instantaneous forward rate} representing the interest rate contractable at time $t \le T$ for the infinitesimal future time interval $(T,T+dA_T]$, see \cite{Filipovic2009} for details and related literature.
The num\'eraire in this market is assumed to be of the from $\exp\big(\int_0^t r(s) dA_s\big)$. 

The term-structure model proposed here is specified by assuming the following structure of the forward rates:
\begin{align}\label{eq:ftT}
	f(t,T) = f(0,T) + \int_0^t a(s,T) dX_s, \quad 0 \le t \le T \le T^*,
\end{align}
where $a$ is a suitable, deterministic function. The first step will be the derivation of a condition on $a$ which renders discounted bond prices local martingales, thus leading to a bond market satisfying a suitable no-arbitrage property, like for example NAFL. 

Consider a filtered probability space $(\Omega,\cF,\bbF,\P)$ satisfying the usual conditions and consider for the beginning a $d$-dimensional, special semimartingale $X$ with semimartingale characteristics $(B,C,\nu)$. As we aim at considering an affine process $X$, with a view on Theorem \ref{prop1} we additionally assume that $X$ has the canonical representation 
\begin{align} X =X_0 + B_t + X^c + x*(\mu-\nu),\label{affine:Xcanonical}\end{align}
where $dB_t = b_t dA_t$, $dC_t = c_t dA_t$ and $\nu(dt,dx)=K_t(dx) d A_t$ and $A$ is deterministic, c\`adl\`ag, increasing with $A_0=0$. We define the left-continuous processes $A(.,T)$, $0< T \le T^*$, by
$$ A(t,T) := \int_{[t,T]} a(s,u) dA_u, \quad 0 \le t \le T, $$ and require the following technical assumption.
\begin{description}
\item[(A1)] Assume that $a:[0,T^*]^2\to \R^d$ is measurable and satisfies 
\begin{align*}
&\Big(\int_.^{T^*} |a_i(.,u)|^2 dA_u\Big)^{\nicefrac 1 2 }\in L(X^i), \quad i=1,\dots,d,\\
&\int_0^{T^*} \int_0^{T^*} |a(t,u)| |dB_t| dA_u < \infty, \quad 0 \le t \le T^*
\end{align*}
where $L(X^i)$ denotes the set of processes which are integrable in the semi-martingale integration sense with respect to the $i$-th coordinate $X^i$ of $X$, $i=1,\dots,d$.
\end{description}

\begin{proposition}\label{affine:prop}
Under {\bf (A1)}, discounted bond prices are local martingales if, and only if
\begin{enumerate}[(i)]
\item $r_t = f(t,t)$ $dA \otimes d\P$-almost surely for $0 \le t \le T^*$, and
\item the following condition holds:
\begin{align}\label{eq:AffineDC}
A(t,T) b_t & = \half A(t,T) c_t A(t,T)^\top  + \int_{\R^d}\Big( e^{A(t,T)x} -1 -A(t,T)x \Big)K_t(dx),
\end{align}
$dA \otimes d\P$-almost surely for $0 \le t \le T \le T^*$.
\end{enumerate}
\end{proposition}
\begin{proof}
The proof follows the classical steps in \cite{HJM}, relying on a stochastic Fubini theorem. First note, that discounted bond prices take the form
\begin{align}
\tilde P(t,T) & = e^{-\int_{(t,T]} f(0,u) dA_u} \exp\bigg( - \int_{(t,T]} \int_{0}^t a(s,u) dX_s dA_u - \int_{(0,t]} r_s dA_s \bigg) \notag\\
&=:P(0,T) \exp(I(t,T)). 
\end{align}
The dynamics of $I$ can be obtained from the dynamics of the forward rates, as
\begin{align*}
	\int_{(t,T]} f(t,u) dA_u &= \int_{(t,T]} f(0,u) dA_u + \int_{(t,T]} \int_0^t a(s,u) dX_s dA_u \\
	&= \int_{(t,T]} f(0,u) dA_u + \int_0^t  \int_{(t,T]} a(s,u) dA_u dX_s  \\
	&= \int_{(t,T]} f(0,u) dA_u + \int_0^t  \int_{[s,T]} a(s,u) dA_u dX_s - \int_0^t  \int_{[s,t]} a(s,u) dA_u dX_s \\
	&= \int_{(t,T]} f(0,u) dA_u - \int_0^t \int_0^u a(s,u)  dX_sdA_u  + \int_0^t  A(s,T) dX_s \\
	&= \int_0^T f(0,u) dA_u - \int_0^t f(u,u) dA_u + \int_0^t A(s,T) dX_s;
\end{align*}
interchange of the integrals is justified under {\bf (A1)} by the Fubini theorem, for example along the lines of \cite{Veraar12,Protter}. The next step is to represent $\exp(I(.,T))=\cE(\tilde I(.,T)) $ as a stochastic exponential $\cE$ on the modified process $\tilde I$ relying on Theorem II.8.10 in \cite{JacodShiryaev}. This theorem yields that
\begin{align*}
 \tilde I(t,T) &= \tilde I(0,T) + I(t,T) + \half \langle I^c(.,T) \rangle_t + (e^x - 1 - x)* \mu^{I(.,T)},
\end{align*}
where $\mu^{I(.,T)}$ denotes the random measure associated to the jumps of $I$, see \eqref{def:muX}. Calculating the above terms under our assumptions together with representation \eqref{affine:Xcanonical} yields that 
\begin{align*}
	d \tilde I(t,T)&= \bigg( - A(t,T) b_t  + \half A(t,T) c_t A(t,T)^\top 
						+ \int_{\R^d}\Big(e^{-A(t,T) x} -1 + A(t,T ) x \Big)K(t,dx) \\
					&\phantom{= \bigg(} + (f(t,t)-r_t)          \bigg) dA_t  + dM_t, \qquad 0 \le t \le T
\end{align*}
with a local martingale $M$. The claim follows by first considering $T=t$, thus yielding (i) and thereafter (ii). For the reverse, observe that (i) and (ii) imply that $\tilde I(.,T)$ is a local martingale, and the claim follows.
\end{proof}

Recall the notion of a good parameter set of the affine semimartingale $X$ from Definition \ref{def:good-parameter}. The following corollary gives a specification of an affine term-structure model in the more classical case, i.e.~when $\gamma=0$.
\begin{corollary}\label{cor5.6}
If {\bf (A1)} holds and  $X$ is a quasi-regular affine semimartingale satisfying the support condition~\ref{def:full_support} and with parameter set $(A,0,\beta,\alpha,\mu)$, 
and if 
\begin{align}\label{eq:DCaffine}
A(t,T) \beta_{i,t} = \half A(t,T) \alpha_{i,t} A(t,T)^\top + \int_{\R^d}\Big( e^{A(t,T)x} -1 -A(t,T)x \Big)\mu_{i}(t,dx),
\end{align}
holds for $i=0,\dots,d$, then the drift condition \eqref{eq:AffineDC} holds.
\end{corollary}
\begin{proof}
The application of Theorem \ref{prop1} yields that $b=\beta_0 + \sum_{i=1}^d X^i_- \beta_i$, with similar expression for $a$ and $K$. Using linearity and \eqref{eq:DCaffine} we immediately obtain \eqref{eq:AffineDC}.
\end{proof}
A reverse version of this result is easily obtained requiring additionally linear independence of certain coefficients, see for example Section 9.3 in \cite{Filipovic2009}.

In the following, we study a variety of extensions of the Vasi\v cek model for incorporating jumps at predictable times. Of course, in a similar manner an extension of the Cox-Ingersoll-Ross model is possible, or one may even extend general stochastically continuous Markov processes in a similar way. 

\begin{example}[The Vasi\v cek model]\label{ex:Vasicek}
We begin by casting the famous Vasi\v cek model in the above framework. The Vas\v cek model is a one-factor Gaussian affine model, where the short rate is the strong solution of the  stochastic differential equation
\begin{align}
dr_t = (\alpha + \beta r_t) dt + \sigma dW_t
\end{align}
with a one-dimensional standard Brownian motion $W$ and $\beta \neq 0$, $\sigma >0$. The bond prices are given in exponential form, such that $P(t,T)=\exp( - \phi(t,T)-\psi(t,T) r_t )$ with $\phi$ and $\psi$ solving certain Riccati differential equation, see \cite{Filipovic2009}, Section 5.4.1, for details. If we embed this approach in our structure given in \eqref{eq:PtT}, we may chose $A_t=t$. The dynamics of $f(t,T)$ in this case will depend also on $R_t:=\int_0^t r_s ds$, such that we utilize the affine process
$$ X_t = (t, R_t, r_t)^\top, \qquad t \ge 0 $$
in \eqref{eq:ftT}. We obtain that $b_t = b_{t}^0 + b_{t}^1 X_t$ with $b_{t}^0=(1, 0, \alpha )^\top$ and $b_{t}^1=(0,1,\beta)$ as well as $c_t=c^0$ where the matrix $c^0$ has vanishing entries except for $c^0_{3,3}=\sigma^2$. The drift condition \eqref{eq:DCaffine} now directly implies that for $A(t,T)=(A^1(t,T),A^2(t,T),A^3(t,T))$
\begin{align}\label{temp1630}
\begin{aligned}
A^2(t,T) & = -\beta A^3(t,T) \\
A^1(t,T) & = (A^3(t,T))^2 \frac{\sigma^2}{2} - \alpha A^3(t,T). \end{aligned}
\end{align}

We have the freedom to choose on component of $A(t,T)$ which we do to match the volatility structure of the Vasi\v cek model, by setting the third component of $A(t,T)$ equal to 
$$ A^3 (t,T) = \beta^{-1}\Big( e^{\beta (T-t)}- 1 \Big). $$

In particular, this choice gives us
\begin{align*}
a^1(t,T) &= \frac{\sigma^2}{\beta} \Big(e^{\beta(T-t)}-1\Big) - \alpha e^{\beta (T-t)},\\
a^2(t,T) &= - \beta e^{\beta(T-t)}, \\
a^3(t,T) &=  e^{\beta(T-t)}.
\end{align*}
It is a straightforward exercise that this specification indeed coincides with the Vasi\v cek model given the explicit expressions for $\phi$ and $\psi$ in Section 5.4.1 in \cite{Filipovic2009}. In a similar manner, all affine term-structure models can be cast in the framework considered in this section.
 \hfill $\diamond$
\end{example}

\begin{example}[A simple Gaussian term structure model]\label{ex5.8}
A review of the above specification points towards the simpler Gaussian model where $X$ is the three-dimensional affine process as above, driven by the Vasi\v cek spot rate, but now we choose
$$ A^3(t,T) = (T-t), $$
such that the parameter $a^3(t,T)=1$ is constant. The drift condition now implies 
\begin{align*}
a^2 & = -\beta,  \\
A^1(t,T) & = (T-t)^2 \nicefrac{\sigma^2}{2} - \alpha (T-t),
\end{align*}
and we obtain a linear term $a^1(t,T) = \sigma^2(T-t) - \alpha$. This Gaussian model is considerably simpler than the Vasi\v cek model, and still has a mean-reversion property (as $X$ has the mean reversion property), but the volatility of the forward rate does not have the dampening factor $e^{\beta(T-t)}$ in the volatility.
 \hfill $\diamond$
\end{example}

Finally, we provide two examples of stochastic discontinuous specifications.

\begin{example}[Example \ref{ex5.8} with discontinuity]
Now we incorporate a stochastic discontinuity at $t=1$ in the above example and let 
$A(t)= t + \ind{t \ge 1}. $ The idea is to introduce a single jump at $t=1$ in the third component and compensate this by a predictable jump in the first coordinate. We begin by describing precisely the model: first, 
$$ dr_t = (\alpha + \beta r_t) dt + \sigma dW_t + dJ_t $$
where $J_t=\ind{t \ge 1} \xi$ with $\xi \sim \cN(0,\gamma^2)$, $\gamma>0$, being independent of $W$. Consider
$$ X_t = (A_t, R_t, r_t)^\top, \qquad t \ge 0, $$
with $R=\int_0^\cdot r_s ds$, as above.
This construction of $X$ implies that for $t \neq 1$, $b_t^0=(1,0,\alpha)^\top$ and $b_t^1=(0,1,\beta)^\top$ while for $t=1$, $b_1^0=(1,0,0)^\top$ and $b_1^1=0$. Moreover,  for $t \neq 1$, $c_t^0 = c_0$ as in the example above,  $c^1_t=0$ and, for $t=1$, we obtain $c_1=0$. 
The kernel $K$ vanishes except for $t=1$ and is  given by 
$K_{1}(dx)=\delta_1(dx^1)\phi(\nicefrac{x^3}{\gamma})dx^3$ where $\delta_1$ is the Dirac measure at point $1$ and $\phi$ is the standard normal density. It does not depend on $\omega$.

As in Example \ref{ex5.8} we specify $a^3=1$, such that $A^3(t,T)=(T-t) + \ind{1 \in [t,T]}$. For $t>1$ the process $A(t,T)$ is exactly as in the previous Example \ref{ex5.8}. 
For the remaining times we again use Corollary \ref{cor5.6}: on the one hand, for $i=1$, the drift condition \eqref{eq:DCaffine} implies that $A^2(t,T) = - \beta A^3(t,T)$ for all $0 \le t \le T$. On the other hand, for $i=0$, the drift condition can be \emph{separated}. Indeed, as $dA_t = dt + \delta_1(dt)$, we obtain, using $\Delta C\equiv 0$, that (for $t=1$)
\begin{align} \label{temp1670}
A(1,T) b_{0,1} &=\int_{\R^d}\Big( e^{-A(1,T)x} -1 +A(1,T)x \Big)K_{0,1}(dx) 
,  \end{align}
and, for $t \neq 1$,
\begin{align} \label{temp1673}
A(t,T) b_{0,t} = \half A(t,T) c_{0,t} A(t,T)^\top. \end{align}
Now Equation \eqref{temp1670} gives
\begin{align}
&&  A_1(1,T) & = e^{-A_1(1,T)+\nicefrac{(A_3(1,T) \gamma)^2}{2}} -1 + A_1(1,T)    \notag \\
\Leftrightarrow &&  A_1(1,T) & =\frac{(A_3(1,T) \gamma)^2}{2}, \label{temp1675} \end{align}
such that $A$ is specified for $t \in [1,T]$. Finally, for $0 \le t < 1$, Equation \eqref{temp1673} implies 
$$ A_1(t,T) = - \alpha A_3(t,T) + \frac{(A_3(t,T)\sigma)^2}{2} $$ 
and we conclude our example. 
 \hfill $\diamond$
\end{example}

\begin{example}[A discontinuous Vasi\v cek model] 
We extend the previous example to the Vasi\v cek model in a more general manner. Consider time points $t_1,\dots,t_n$ which correspond to stochastic discontinuities. Moreover, assume that 
$$ dr_t = (\alpha + \beta r_t) dt + \sigma dW_t + dJ_t $$
were 
$$ J_t = \sum_{i=1}^n \ind{t_i \le t} \xi_i, \quad t \ge 0,$$
with $\xi_i$ being i.i.d.~$\sim \cN(0,\gamma^2)$, being independent of $W$. Let $A_t = t + \sum_{i=1}^n \ind{t_i \le t}$ and consider as above $X = (A, R, r)$. Again, for $t \not \in \{t_1,\dots,t_n\}$, $b_{0,t}=(1,0,\alpha)^\top,$  $b_{1,t}=(0,1,\beta)^\top$, and $c_{0,t} = c_0$ while for $t=t_i$, $b_{0,t_i}=(1,0,0)^\top$, $b_{1,t_i}=0$ and $c_{t_i}=0$. Moreover,
$$ K_{t}(dx)=\ind{t \in \{t_1,\dots,t_n\}}\delta_1(dx^1)\phi(\nicefrac{x^3}{\gamma})dx^3. $$
We begin by specifying $a^3(t,T)=e^{\beta (T-t)}$ as in Example \ref{ex:Vasicek}, such that
$$ A^3(t,T) = \beta^{-1}\Big( e^{\beta (T-t)}- 1 \Big) + \sum_{i=1}^n \ind{t_i \in [t,T]}. $$
Again, we separate the drift condition in continuous and discontinuous part with the aid of Corollary \ref{cor5.6} yielding directly $A^2(t,T) = -\beta A^3(t,T)$ and $A^1(t,T)  = (A^3(t,T))^2 \frac{\sigma^2}{2} - \alpha A^3(t,T)$, for $t \in [0,T] \backslash \{t_1,\dots,t_n\}$, compare Equation \eqref{temp1630}. It remains to compute $A(t_i,T)$ for $t_i \le T$. In this regard, we obtain as in \eqref{temp1675} that
\begin{align} \label{temp1676}
A(t_i,T) & =\frac{(A_3(t_i,T) \gamma)^2}{2}, \quad i=1,\dots,n,
\end{align}
such that the discontinuous Vasi\v cek model is fully specified.
 \hfill $\diamond$
\end{example}

\begin {appendix}
\section{Measure differential equations}\label{General-measure-DE}
This section recalls and extends some notions and statements about measure differential equations (somtimes also referred to as Stieltjes differential equations) for the special cases needed in this article.

Let $A$ be an increasing function on $\R_{\geq0}$ with left limits and $F\colon \R_{\geq0} \times \mathcal{U}\rightarrow \mathcal{U}$, where the space $ \cU$ is  defined in Equation \eqref{eq:defU}. Assume $F(\cdot,g\brac{\cdot})$ is $A$-integrable on some interval $I\subset \R_{\geq 0}$ for all functions $g\colon \R_{\geq 0} \rightarrow \mathcal{U}$ of bounded variation. We consider the equation
\begin{equation}
\frac{dg(t)}{dA_t}=-F\brac{t,g(t)},\quad g(T)=u,\label{eq:measure-differential-equation}
\end{equation}
$dg/dA$ denotes the Radon-Nikodym derivative of the measure induced by $g$ with respect to the measure induced by $A$. We now recall the definition of a solution to a measure differential equation from \cite{SharmaDas1972} that we adopt in this article.

\begin{definition}Let $S$ be an open connected set in $\mathcal{U}$ and $T\in I$. A function $g\brac{\cdot}=g\brac{\cdot,T,u}$ will be called a solution of \eqref{eq:measure-differential-equation} through $\brac{T,u}$ on the interval $I$ if $g$ is right-continuous, of bounded variation, $g(T)=u$ and the distributional derivative of $g$ satisfies \eqref{eq:measure-differential-equation} on $\brac{\tau,T}$ for any $\tau<T$ in $I$.
\end{definition}

\begin{remark}
Assume $F\brac{t,g(t)}$ is integrable with respect to the Lebesgue-Stieltjes measure $dA$ for each function $g$ of bounded variation. Equivalently to the above definition $g$ is a solution of \eqref{eq:measure-differential-equation} through $\brac{T,u}$ on $I$ if and only if it satisfies the integral equation
\begin{equation}
g(t)=u+\int_{(t,T]}F\brac{s,g(s)}dA_s,
\label{eq:integral-measure-de}
\end{equation}
see \cite{SharmaDas1972} for more details. 
\end{remark}
We are now going to state and prove a modification of the existence and uniqueness result for measure differential equations in \cite{Sharma1972}. Define 
\begin{eqnarray*}
\Omega_b=\set{u\in \mathcal{U}\vert \norm{u}< b}
\end{eqnarray*}

\begin{theorem}\label{thm:ex-unique-de}
Suppose the following conditions hold
\begin{enumerate}[(i)]
\item there exists an $A$-integrable function $w$ such that 
\begin{equation}
\norm{F(t,u)}\leq w\brac{t}
\label{eq:F-integrability}
\end{equation}
uniformly in $u\in\Omega_b$;
\item $F$ satisfies a Lipschitz condition in $u$, i.e. there exists an $A$-integrable Lipschitz constant $L$ such that 
\begin{equation*}
\norm{F\brac{t,u_1}-F\brac{t,u_2}}\leq L(t)\norm{u_1-u_2}
\end{equation*}
for all $u\in \Omega_b$.
\end{enumerate}
Then there exists a unique solution $g$ of \eqref{eq:measure-differential-equation} on some interval $(T-a,T]$, $a>0$, satisfying the terminal condition $g(T)=u$.
\end{theorem}
\begin{proof} First note that we have the following equation for the jumps of a solution $g$ to \eqref{eq:measure-differential-equation}, for all $t\in\set{t\in\R_+\vert \Delta A_t\neq 0}$,
\begin{equation}
\label{eq:jump-equation-DE}
\Delta g(t)=-F\brac{t,g(t)}\Delta A_t.
\end{equation}
With $\Delta g(t)=g(t)-g(t-)$ this is an explicit equation for the left limit of $g$, hence we can assume that $A$ has no jump at the terminal time $T$, as we can simply compute $g(T-)$ from the terminal value and start from there instead. Even with time-varying Lipschitz constant the proof of Theorem 1 in \cite{Sharma1972} is valid with small adjustments: $A$ is increasing and c\'adl\'ag. Therefore there exists $r\in\brak{0,T}$ such that

\begin{equation*}
\int_{(r,T]}L\brac{s}dA_s < 1
\end{equation*}
and 
\begin{equation}
\label{eq:boundedness-for-K}
k\coloneqq\norm{u}+\int _{(r,T]} w\brac{s} dA_s < b.
\end{equation}

Denote the space of c\'adl\'ag functions $f$ on $(r,T]$ with terminal value $f(T)=u$ and total variation $\Norm{f}\leq k$ by $\Lambda$ and consider the mapping

\begin{equation*}
K f(t)=u-\int_{(t,T]}F\brac{s,f(s)}dA_s,\quad t\in(r,T].
\end{equation*}

It follows from condition (i) and equation \eqref{eq:boundedness-for-K} that $K$ maps $\Lambda$ into itself. From the Lipschitz condition on $F$ we obtain

\begin{equation*}
\Norm{K f_1-K f_2}\leq \Norm{f_1-f_2}\int_{(r,T]}L(s)dA_s.
\end{equation*}

Hence, $K$ is a contraction on $\Lambda$ - a closed subspace of the space of c\`adl\`ag functions with bounded variation.This implies the existence of a unique fixed point of $K$, which is the desired local solution of \eqref{eq:F-integrability}.
\end{proof}

\begin{example}[The linear equation]\label{ex:linear-equation} Let $A$ as above and $L \in L_1(dA)$ with $L(t) \Delta A_t \ge -1$ for all $t \ge 0$.
Consider the linear measure equation
\begin{equation}
\frac{d}{dA_t}\phi(t)=-L(t) \phi(t)\quad \phi(T)=\phi_T
\label{eq:linear-abstract-de1}
\end{equation}
on $[0,T]$. The process $ \tilde{A}_t\coloneqq \int_{\brak{0,t}} L\brac{s} dA_s$ has finite variation and thus we can apply \cite[Theorem I.4.61]{JacodShiryaev} and especially equation I.4.63 to obtain that the unique, c\`adl\`ag solution to the linear equation \eqref{eq:linear-abstract-de1} is given by $\phi(t) = \phi_T \cE_t^T(L\,dA)$ where

\begin{equation}
\begin{split}
\cE_t^T(L\,dA) := & e^{\int_t^{T} L(s)dA_s} \prod_{s\in(t,T]}\brac{1+L(s)\Delta A_s}e^{-L(s)\Delta A_s} \notag
\\
=& e^{\int_t^{T} L(s)dA_s^c}\prod_{s\in(t,T]} \brac{1+L(s)\Delta A_s}.\label{eq:pseudo_exp}
\end{split}
\end{equation}
\end{example}

\begin{proposition}\label{prop:comparison}
Let $f,g$ be right-continuous and absolutely continuous w.r.t.~$A$. If the following conditions hold
\begin{enumerate}[(i)]
\item $f(T)\leq g(T)$,
\item $\frac{d}{dA_t}f(t)=-F\Brac{t,f(t)}$ and $\frac{d}{dA_t}g(t)=-G\Brac{t,g(t)}$ on $ I = [0,T]$, where $F,G$ are locally Lipschitz continuous in the second variable with A-integrable Lipschitz constants, and 
\item $F(t,u)\leq G(t,u)$ for all $t\in I$,
\end{enumerate}
then $f(t)\leq g(t)$ for all $t\in I$.
\end{proposition}

\begin{proof}
Suppose the conclusion of the proposition does not hold. Let $w=f-g$. Then exists an interval $I' = [t_0,t_1)$ such that $w$ is positive and continuous on $I'$ and $w\brac{t_1}\leq 0$. Two cases can occur: $\Delta A_{t_1}=0$ or $\Delta A_{t_1}\neq 0$.

Consider first the case when there is no jump at $t_1$. From condition (ii) and (iii) we obtain on $(t_0,t_1]$ that 
\begin{align*} 
\frac{dw(t)}{dA_t} &= G(t,g(t)) -F(t,f(t)) \ge G(t,g(t)) - G(t,f(t)) \ge - L_t w(t),
\end{align*}
where $L_t$ is the Lipschitz constant of $G(t,.)$ on the relevant domain. Consider the function $W(t)=w(t) \exp \left(-\int^{t_1}_t L_s dA_s \right)$ on $(t_0,t_1]$. $W$ is absolutely continuous w.r.t. $A$ and continuous. Furthermore 
\begin{equation*}
\frac{dW(t)}{dA_t}=\brac{\frac{dw(t)}{dA_t}+L_t w(t)}e^{-\int_t^{t_1}L_s dA_s} \ge 0,\quad t\in(t_0,t_1].
\end{equation*}
Together with $w\brac{t_1} \le 0$ it follows that $w(t)\leq 0$ for all $t\in (t_0,t_1]$ contradicting the assumption.
Second, if we have a jump at $t_1$, i.e. $\Delta w\brac{t_1}\neq 0$, we immediately get $\Delta w\brac{t_1}<0$ and therefore
\begin{eqnarray*}
0>\Delta w\brac{t_1}&=&-\brac{F\brac{t_1,f(t_1)}-G\brac{t,g(t_1)}}\Delta A_{t_1}
\\
&\geq& -L_{t_1}w\brac{t_1}\Delta A_{t_1}.
\end{eqnarray*}
Hence, $w(t_1)>0$; a contradiction.
\end{proof}

\end{appendix}

\end{document}